\documentclass{article}
\usepackage{color}
\usepackage{amsfonts}
\usepackage{tikz}
\usepackage{amssymb}
\usepackage{amsmath}
\usepackage{setspace}
\usepackage{circuitikz}
\usepackage{amsthm}
\usepackage[mathscr]{euscript}
\usepackage{hyperref}
\usepackage{geometry}
\newtheorem{theorem}{Theorem}[section]

\newtheorem{definition}{Definition}[section]
\newtheorem{lemma}{Lemma}[section]
\newtheorem{proposition}{Proposition}[section]
\newtheorem{remark}{Remark}[section]

\numberwithin{equation}{section}

\newcommand{\m}{\mathbb}
\newcommand{\ml}{\mathcal}
\newcommand{\p}{\partial}
\newcommand{\pprec}{\prec\!\!\!\prec}

\makeatletter
\usepackage{graphicx} 

\allowdisplaybreaks[2] 

\begin{document}

\setstretch{1.25}
\title{Rough Stochastic Viscous Camassa-Holm equation}
\author{Xiaoxin $\mbox{Chen}^1$ \footnote{E-mail: chenxx233@mail2.sysu.edu.cn},\quad
    Yiyao $\mbox{Lian}^1$ \footnote{E-mail: lianyy7@mail2.sysu.edu.cn},\quad
	Zhaoyang $\mbox{Yin}^{1}$\footnote{E-mail: mcsyzy@mail.sysu.edu.cn},\quad
    Shuolin $\mbox{Zhang}^1$ \footnote{E-mail: zhangshlin9@mail2.sysu.edu.cn}\\
$^1\mbox{School}$ of Science,\\ Shenzhen Campus of Sun Yat-sen University, Shenzhen 518107, China}

\date{}
\maketitle
\hrule

\begin{abstract}
In this paper, we establish well-posedness results for the viscous Camassa--Holm equation driven by the additive Gaussian noise $(1-\partial_x^2)^{-\eta}\xi$, where $\xi$ denotes the space-time white noise and $\eta>-1/12$. This formulation includes spatially colored noise for $\eta>0$, space-time white noise for $\eta=0$, and rough noise for $-1/12<\eta<0$. More precisely, we establish a pathwise solution theory via approximation by equations driven by regularized noises. The main difficulty arises from the singular interactions between the nonlinear terms and the irregular stochastic convolution. By introducing suitable enhanced stochastic objects and applying paracontrolled techniques, we obtain local well-posedness for the full range $\eta>-1/12$. Furthermore, when $\eta>1/4$, we establish global well-posedness by deriving suitable energy estimates.

\vspace*{5pt}
\noindent {\it 2020 Mathematics Subject Classification}: 
	
	\vspace*{5pt}
\end{abstract}

\tableofcontents

\section{Introduction}
In this paper, we consider the viscous Camassa-Holm equation with spatially colored Gaussian noise. First, we introduce the viscous Camassa-Holm equation as follows: 
\begin{equation}\label{eq;viscous CH}
    u_t-u_{xxt}-\epsilon(u_{xx}-u_{xxxx})+3uu_x=2u_xu_{xx}+uu_{xxx},\quad t>0,\\
\end{equation}
where $u$ is the fluid velocity and $\epsilon> 0$ is the viscosity constant. The above equation admits the following equivalent form:
\begin{equation}\label{eq;viscous CH equivalent form}
    u_{t}+uu_x-\epsilon u_{xx}=P(D)(u^2+\frac{1}{2}u_x^2),
\end{equation}
where $P(D)\triangleq -\p_x(1-\p_x^2)^{-1}$.

\par The viscous Camassa–Holm equation, which reduces to the classical Camassa–Holm (CH) equation when $\epsilon=0$, arises as a natural generalization incorporating viscous effects and occupies a unique and important position in the theory of nonlinear partial differential equations. The classical Camassa–Holm  equation takes the form  
\begin{equation}\label{eq;no noise_Camassa-Holm}
    u_t-u_{txx}+ 3uu_x=2u_xu_{xx}+uu_{xxx},
\end{equation}
which is a fundamental nonlinear partial differential equation arising in fluid dynamics and integrable systems, and has attracted increasing attention in recent years. The variable $u(t,x)$ in \eqref{eq;no noise_Camassa-Holm} denotes the fluid velocity in the $x$-direction at time $t$, or equivalently, the height of the water's free surface above a flat bottom. The CH equation was initially proposed by Fuchssteiner and Fokas \cite{Fokas1981} in 1981 in connection with hereditary symmetries. Later, Camassa and Holm \cite{Camassa1993} derived it explicitly as a shallow water wave model in 1993. The CH equation is well known for its rich mathematical structure. It is completely integrable \cite{Camassa1993,Con2,Con3}, admits a bi-Hamiltonian formulation \cite{Con4,Fokas1981}, and supports peaked soliton solutions (peakons) of the form  $ce^{-\vert x-ct\vert}$ \cite{Camassa1993}. Moreover, the equation captures wave-breaking phenomena \cite{Beals2000,Con5,Whitham1999}, and its local well-posedness in Sobolev and Besov spaces has also been established \cite{Con8,Danchin2001,Danchin2003}. 

\par The viscous Camassa-Holm equation  has attracted considerable attention over the past two decades, with extensive studies devoted to various aspects of its mathematical theory. In \cite{Foias2001}, the Camassa–Holm equation is regarded as the nonlinearly dispersive Navier–Stokes-$\alpha$ model for incompressible turbulence. This physical interpretation naturally motivates the study of its viscous counterpart, as the inclusion of viscosity aligns with the dissipative mechanisms present in real fluid flows. A natural multi-dimensional extension of the viscous Camassa–Holm equation \eqref{eq;viscous CH} is provided by the Navier–Stokes-$\alpha$ equation, which reads
\begin{equation}\label{eq;NS-alpha}
\left\{\begin{array}{l}
\partial_t v+u\cdot\nabla v+v\cdot\nabla u^{T}+\nabla \pi-\nu\Delta v=0,\\
v=u-\alpha^2\Delta u,\quad 
\mathrm{div}v=0
\end{array}\right.
\end{equation}
with the initial data $v_0= u_0-\alpha^2\Delta u_0$. This system plays a central role in modeling turbulent flows at large scales, as it provides a mathematically tractable framework that incorporates both viscous dissipation and nonlinear dispersion. For further studies on the Navier–Stokes-$\alpha$ system, we refer to \cite{Chen1998,Foias2001,Foias2002JDE}.  \par
\par Nevertheless, deterministic models may not fully reflect the intrinsic uncertainties and random fluctuations arising in complex fluid systems.  In this paper, we focus on viscous CH equation with stochastic forcing. Before showing the main results, we introduce stochastic and SPDE theory.
\subsection{Noise perturbation}
\par The introduction of stochastic perturbations provides a natural framework for describing random fluctuations that cannot be captured by deterministic models. Different types of stochastic forcing have been considered in fluid dynamics. For instance, multiplicative noise can be introduced to model random transport effects and energy exchange mechanisms. In \cite{holm2015variational}, Holm derived stochastic partial differential equations for ideal fluid dynamics from stochastic variational principles, where the Stratonovich stochastic processes represent unresolved fluctuations in the fluid motion. More general multiplicative noises of the form $h(t,u,\p_xu)dW$ have also been investigated, see, e.g.,\cite{rohde2021stochastic,miao2021well,holden2022global}.  Additive stochastic forcing has also been studied in various evolution equations. For example, additive noise perturbations for KdV-type equations have been considered in \cite{chen2012well}. However, the viscous Camassa--Holm equation with spatially singular stochastic forcing has not been fully understood. Motivated by these considerations, we investigate the viscous Camassa--Holm equation subject to white noise perturbations on the torus $\mathbb{T}$. Without loss of generality, we take the viscosity coefficient $\epsilon=1$. Namely, we consider the viscous Camassa-Holm equation with spatially colored Gaussian noise  perturbation on $\m T$:
\begin{equation}\label{eq;CH;diffusion}
			u_{t}+uu_x-\Delta u=P(D)(u^2+\frac{1}{2}u_x^2)+(1-\p_x^2)^{-\eta}\xi,
	\end{equation}
where $\xi$ is space-time white noise satisfying 
\begin{equation}\label{eq;covariance;xi}
    \m E[\xi(t,x)\xi(s,y)]=\delta_0(t-s)\delta_0(x-y)
\end{equation}
and the parameter $\eta$ quantifies the roughness of  noise in the spatial variable. For $\eta>0$, $(1-\p_x^2)^{-\eta}\xi$ is called colored or regularized noise; for $\eta=0$, $(1-\p_x^2)^{-\eta}\xi=\xi$ is white noise; for $\eta<0$, it is called rough noise. 
\par The motivation of this work is twofold. First, it is inspired by the fluctuation-dissipation theorem, which reveals that fluctuations and dissipation are two manifestations of the same underlying physical mechanism; see, e.g., \cite{kubo1966fluctuation,prost2009generalized,stratonovich2012nonlinear}. This fundamental principle suggests that dissipative systems should naturally be coupled with stochastic perturbations in order to capture the corresponding fluctuations induced by the dissipative effects.
\par Second, our study is motivated by the stochastic Navier--Stokes equations with white noise forcing, which provide a fundamental framework for incorporating random fluctuations into fluid dynamics; see, e.g., \cite{flandoli1995ergodicity,hairer2006ergodicity,hairer2024global}. In particular, the multidimensional viscous Camassa--Holm equation is closely related to the Navier-Stokes-$\alpha$ model \eqref{eq;NS-alpha}.
\par The main difficulty in the analysis of the stochastic Camassa--Holm equation comes from the interaction between the rough stochastic forcing and the nonlinear structure of the equation. Due to the low regularity of the stochastic convolution, nonlinear terms such as $u_x^2$ and $uu_x$ cannot generally be interpreted in the classical sense. This difficulty is also encountered in other fundamental examples of singular SPDEs, such as the Rough Burgers and KPZ equations.
\subsection{Singular SPDEs}
\par Without the nonlinearity $P(D)(u^2+\frac{1}{2}u_x^2)$, the equation becomes the viscous Burgers equation with spatially colored Gaussian noise 
\begin{equation}
    \p_t u+uu_x- \Delta u=(1-\p_x^2)^{-\eta}\xi.
\end{equation}
 The choice $\eta=-1/2$ corresponds to a noise with the same spatial scaling as the derivative of white noise appearing in the classical Rough Burgers equation \cite{gubinelli2017kpz}
 \begin{equation}\label{eq;rough Burgers}
     \p_tu-\Delta u=\p_x(u^2)+\p_x\xi,
 \end{equation}
 where the solution is the derivative of the solution of KPZ equation \cite{hairer2013solving,gubinelli2017kpz}
 \begin{equation}
     \p_t u-\Delta u=u_x^2+\xi.
 \end{equation}
\par For these fundamental examples of singular SPDEs, one of the main difficulties lies in the ill-defined nonlinear terms caused by the low regularity of the stochastic forcing. Due to the singular nature of the noise, classical analytical techniques for nonlinear PDEs are no longer applicable, and a suitable solution theory was not available for a long period of time. A major breakthrough was achieved by Hairer, who introduced the notion of pathwise solutions together with a detailed approximation procedure in \cite{hairer2013solving}, and developed the theory of regularity structures in \cite{hairerTheoryRegularityStructures2014}. This theory provides a general framework for constructing solutions to a broad class of singular SPDEs.
\par Meanwhile, Gubinelli and collaborators introduced the paracontrolled calculus in \cite{gubinelli2015paracontrolled,gubinelli2017kpz}. Although this approach is less general than Hairer's regularity structures, it has the advantage of a more explicit construction and greater flexibility for equations with specific algebraic structures. Moreover, it establishes a close connection between harmonic analysis and rough path theory.
\par To briefly illustrate the basic strategy for treating singular equations, we consider the rough Burgers equation \eqref{eq;rough Burgers}. We first study its regularized version
\begin{equation}
\p_t u^\sigma-\Delta u^\sigma
=\p_x\bigl((u^\sigma)^2\bigr)+\p_x\xi_\sigma,
\end{equation}
where $\xi_\sigma$ denotes a spatial mollification of the space-time white noise $\xi$. In one spatial dimension, the noise has regularity
$$
    \xi\in \ml C_t^{-1/2-}\ml C_x^{-1/2-}.
$$
To isolate its leading singular component, let \(X^\sigma\) be the stochastic convolution solving
\begin{equation}
\p_tX^\sigma-\Delta X^\sigma=\p_x\xi_\sigma,\quad X^\sigma(0)=0.
\end{equation}
Writing $$ v^\sigma=u^\sigma-X^\sigma, $$
we obtain the random PDE
\begin{equation}
\p_tv^\sigma-\Delta v^\sigma
=\p_x\bigl(v^\sigma+X^\sigma\bigr)^2.
\end{equation}
Although this transformation removes the explicit noise from the equation, it does not by itself resolve the singularity. Indeed, \(X^\sigma\) converges in \(C_T\ml C^\alpha\) only for \(\alpha<-1/2\), so products such as \((X^\sigma)^2\) are not classically well defined in the limit \(\sigma\to0\).

The central idea of rough-path-inspired approaches to singular SPDEs is to separate the leading singular components of the solution through an expansion in higher-order stochastic objects. These objects may converge, after suitable renormalization, owing to cancellations arising from the probabilistic structure of the noise. This motivates a representation of the limiting solution of the form in \textbf{``high order expansion''}
\begin{equation}
u=\sum_{\tau\in\ml T'}c(\tau)X^\tau+v,
\end{equation}
where the stochastic processes \(X^\tau\), \(\tau\in\ml T'\), collectively form an enhanced stochastic process encoding the singular interactions generated by the noise.

To control the remaining singular products, such as \(\p_x(v^\sigma X^\sigma)\) in the regularized equation, one employs the theory of paracontrolled distributions introduced by Gubinelli, Imkeller, and Perkowski. Combining the stochastic expansion with a paracontrolled ansatz allows one to establish convergence in probability of the regularized solutions \(u^\sigma\) to a limiting pathwise solution \(u\) in a suitable function space, such as \(C_T\ml C^\alpha\).

More precisely, the limiting remainder \(v\) is sought in the paracontrolled form
\begin{equation}\label{eq;paracontrolled ansatz}
v=u'\prec\ml Q+u^\sharp,
\end{equation}
where the reference process \(\ml Q\) and the coefficient \(u'\) are chosen so that \((\p_t-\Delta)(u'\prec\ml Q)\) matches the least regular terms in the equation for \(v\). Subtracting this contribution yields an equation for \(u^\sharp\) with a more regular right-hand side. Thus, \(u'\prec\ml Q\) captures the leading singular behavior of \(v\), while \(u^\sharp\) has higher regularity. Even though $u'\prec \ml Q$ is also gains low regularity, we can using the ``paracontrolled commutator" in \cite{gubinelli2015paracontrolled} as
\begin{equation}
    C(u',\ml Q,X)=u'\prec\ml Q\circ X-u'(\ml Q\circ X)
\end{equation}
which is well-defined for
Together with the enhanced stochastic data and the paracontrolled commutator estimates
, this decomposition makes it possible to define the otherwise ill-defined resonant products arising in \(\p_x(vX)\).

The parameter \(\eta\) determines the regularity of the stochastic convolution and thus plays a crucial role in the analytical difficulty of the problem. Let \(X\) denote the stochastic convolution solving

$$
(\partial_t-\Delta)X=(1-\partial_x^2)^{-\eta}\xi,
\qquad X(0)=0.
$$

By the regularizing effect of the heat semigroup, for any \(\kappa<\frac32+2\eta\), one has

$$
X\in C^{\kappa-1}\cap H^{\kappa-1},
\qquad
\partial_xX\in C^{\kappa-2}\cap H^{\kappa-2}.
$$

Accordingly, the analysis naturally splits into the following four regularity regimes:
\begin{itemize}

\item When \(\eta>\frac14\), one may choose \(\kappa>2\), so that \(\partial_xX\) has positive spatial regularity. The nonlinear interactions can therefore be handled by classical product estimates.

\item When \(\frac18<\eta\leq\frac14\), \(\partial_xX\) has negative regularity, and the resonant product \(\partial_xX\circ\partial_xX\) is no longer classically well defined. Nevertheless, this resonant term can be constructed by exploiting stochastic cancellations arising from the Gaussian structure of \(X\).

\item When \(0<\eta\leq\frac18\), the stochastic cancellation at the previous level no longer provides sufficient regularity to control all nonlinear interactions. To capture the resulting higher-order singularities, we introduce the bilinear evolution operator

$$
\widetilde{B}(f,g)(t)
\triangleq
\int_0^t e^{(t-s)\Delta}
\partial_x(1-\partial_x^2)^{-1}\big(f(s)g(s)\big)\,\mathrm{d}s.
$$

In particular, the second-order stochastic term \(\widetilde B(Y_x,Y_x)\), where \(Y\) denotes the stationary stochastic convolution as 
\begin{equation}\label{eq;stationary}
    \p_t Y-\Delta Y=(1-\p_x^2)^{-\eta}\xi,\quad Y(0)=\int_{-\infty}^0e^{-s\Delta}(1-\p_x^2)^{-\eta}\xi(s)\mathrm ds,
\end{equation}
has to be incorporated into the enhanced stochastic data. This higher-order expansion allows us to isolate and control the additional singular interactions arising in this regime.

\item When \(-\frac1{12}<\eta\leq0\), mixed interactions between the stochastic component and the solution remainder, in particular \(\partial_xv\cdot\partial_xX\), may also become ill defined. We therefore introduce the paracontrolled ansatz \eqref{eq;paracontrolled ansatz}. For \((u',\mathcal Q,X)\in W^{\kappa-1}\times W^\kappa\times W^{\kappa-1}\), the key commutator satisfies

$$
C(u',\mathcal Q_x,X_x)\in W^{3\kappa-4}.
$$

This commutator is well defined provided that \(3\kappa-4>0\). Since \(\kappa\) can be chosen below \(\frac32+2\eta\), this condition is compatible with the regularity of \(X\) precisely when \(\eta>-\frac1{12}\), which explains the lower threshold in our local well-posedness theory.

\end{itemize}
Thus, as \(\eta\) decreases, the analysis passes successively from classical product estimates to stochastic cancellations, higher-order stochastic enhancements, and finally a paracontrolled formulation.

\par  Depending on the regime of \(\eta\), different techniques are required to handle the interaction between the stochastic forcing and the nonlinear structure of the viscous Camassa-Holm equation. The detailed regularity analysis is presented in Section 2. In particular, for the full range $\eta>-1/12$, we obtain the following results.
\subsection{Main results}
\begin{theorem}\label{th;main results}
    Let $u_0\in H^s$, $\eta>-1/12$, $Y(0)=\int_{-\infty}^0e^{-s\Delta}(1-\p_x^2)^{-\eta}\xi(s)\,\mathrm ds$, where $s$ and $\eta$ satisfy the following conditions:
    \begin{itemize}
        \item {$\eta\in(\frac14,\infty)$, $\kappa\in(2,\frac{3}{2}+2\eta)$, $s\in[1,\kappa)$, $u(0)=u_0$\textup{;}}
        \item {$\eta\in(\frac18,\frac14]$, $\kappa\in(\frac{3}{2},\frac{3}{2}+2\eta)$, $s\in[3-\kappa,\kappa)$, $u(0)=u_0$\textup{;}}
        \item {$\eta\in(0,\frac18]$, $\kappa\in(\frac{3}{2},\frac{3}{2}+2\eta)$, $s\in[3-\kappa,\kappa)$, $u(0)=u_0+Y(0)$\textup{;}}
        \item {$\eta\in(-1/12,0]$, $\kappa \in (\frac{5}{4}-\eta,\frac{3}{2}+2\eta)$, $s=2\kappa-1$, $u(0)=u_0+Y(0)$.}
    \end{itemize}
   Let $\xi_\sigma$ be spatial mollification of the space-time white noise \(\xi\). Then, for any $\sigma>0$, there exists unique $u^\sigma$ and $T_\sigma$ such that $u^\sigma$ is the solution of the following approximation equation
    \begin{equation}\label{eq;CH;diffusion;app}
		\left\{\begin{array}{l}
			u^\sigma_{t}+u^\sigma u^\sigma_{x}- \Delta u^\sigma=P(D)((u^\sigma)^2+\frac{1}{2}(u^\sigma_x)^2)+(1-\p_x^2)^{-\eta}\xi_\sigma,\\
			u|_{t=0}=u(0),
		\end{array}\right.
	\end{equation}
on $[0,T^\sigma]$ and $\lim_{\sigma\to0} T^\sigma=T>0$.   Moreover, there exists $u\in C([0,T);H^{\kappa-1})$ such that $\Vert u^\sigma-u\Vert_{C_{T_\sigma\wedge T}H^{\kappa-1}}\to 0$ in probability sense as $\sigma\to0$.
\end{theorem}

\begin{remark}
The main difference between \(X\) and \(Y\) is that \(Y\) is chosen to be the stationary solution of the linear stochastic heat equation
\begin{equation}
\partial_t Y-\Delta Y=(1-\partial_x^2)^{-\eta}\xi,
\qquad
Y(0)=\int_{-\infty}^{0}e^{-s\Delta}
(1-\partial_x^2)^{-\eta}\xi(s)\,\mathrm ds.
\end{equation}
In particular, \(Y\) is stationary in time and admits an explicit covariance structure. This stationarity is particularly convenient for the construction and analysis of higher-order stochastic objects, since all such objects can be built from a single stationary reference process. In particular, it provides a unified framework for constructing the higher-order enhancements required in the non-Gaussian rough-path setting.
\end{remark}

Singularity of nonlinearity lead the solution should be work in weak sense, the regular case allow us to study it in mild sense, we also obtain the global results in $H^s$ space by energy estimate as follows
\begin{theorem}\label{thm;global solution}
    Let $\eta>1/4$ and $u_0\in  H^1$ satisfying $\int_{\mathbb{T}}u(x)\,\mathrm{d}x=0$, then \eqref{eq;CH;diffusion} has a global solution in $L_T^\infty  H^1\cap L_T^2 H^2$.
\end{theorem}
\begin{remark}
    It is worth emphasizing that the parameter range for global well-posedness is more restrictive than that for local well-posedness. While the local theory can be established for the range $\eta>-1/12$, the global analysis requires additional regularity of the stochastic convolution in order to obtain uniform a priori estimates in $H^1$. Consequently, we restrict our global well-posedness result to the regime $\eta>1/4$. It is worth mentioning that extending the global well-posedness theory to $\eta\leq1/4$ regime remains a challenging problem. In such regimes, the stochastic forcing exhibits stronger irregularity, and the classical energy estimates (See Proposition \ref{esti;energy esti}) may no longer be sufficient to control the nonlinear interactions. A possible approach would require combining rough path techniques with suitable energy methods to handle the interplay between the rough stochastic perturbation and the nonlinear Camassa--Holm structure. We leave this interesting direction for future investigation.
\end{remark}

\subsection{Organization}
    \par The remainder of this paper is organized as follows. In Section \ref{sec;pre}, we introduce the preliminary materials and analytical tools used throughout the paper. Section \ref{sec;enhanced data} is devoted to the construction of the enhanced stochastic process. In Section \ref{sec;wellposedness}, we establish the local well-posedness in the regime $\eta>-1/12$. Section \ref{sec;global well-posedness} is devoted to the global well-posedness in the regular regime $\eta>1/4$. 

\section{Preliminaries}\label{sec;pre}
In this section, we first recall the Littlewood-Paley decomposition and the Bony decomposition, together with a modified version thereof.

\begin{proposition}\cite{bahouriFourierAnalysisNonlinear2011}\label{prop0}
Let $\mathcal{B}=\{\xi\in\mathbb{R}^d: |\xi|\leq \frac 43\}$ and $\mathcal{C}=\{\xi\in\mathbb{R}^d: \frac 34 \leq|\xi|\leq \frac 83\}$. There exist two smooth, radial functions $\chi$ and $\varphi$, valued in the interval $[0,1]$, belonging respectively to $\mathcal{D(B)}$ and $\mathcal{D(C)}$, such that
$$
	\forall\ \xi\in\mathbb{R}^{d},\ \chi(\xi)+\sum_{j\geq0}\varphi(2^{-j}\xi)=1, 
$$
$$
	|j-j^{\prime}|\geq2\Rightarrow\mathrm{Supp~}\varphi(2^{-j}\cdot)\cap\mathrm{Supp~}\varphi(2^{-j^{\prime}}\cdot)=\emptyset, 
$$
$$
	j\geq1\Rightarrow\mathrm{Supp~}\chi\cap\mathrm{Supp~}\varphi(2^{-j}\cdot)=\emptyset.
$$
Let $u\in\mathcal{S}'$. Defining
$$
\Delta_{j}u\triangleq0\ \text{if}\ j\leq-2,\ \Delta_{-1}u\triangleq\chi(D)u=\mathcal{F}^{-1}(\chi\mathcal{F}u),
$$
$$
\Delta_{j}u\triangleq\varphi(2^{-j}D)u=\mathcal{F}^{-1}(\varphi(2^{-j}\cdot)\mathcal{F}u)\ \text{if}\ j\geq0,
$$
$$
S_j u\triangleq\sum_{j^{\prime}\leq j-1}\Delta_{j^{\prime}}u,
$$
we have the following Littlewood-Paley decomposition:
$$u=\sum_{j\in\mathbb{Z}}\Delta_{j}u\quad in\ \mathcal{S}'.$$
\end{proposition}

\begin{definition}\cite{bahouriFourierAnalysisNonlinear2011}\label{Besov Space}
    Let $s\in\mathbb{R}$ and $(p,r)\in[1,\infty]^2$. The nonhomogeneous Besov space $B^s_{p,r}$ consists of all $u\in\mathcal{S}'$ such that
    $$
    \|u\|_{B_{p,r}^s}\triangleq\left\|(2^{js}\|\Delta_ju\|_{L^p})_{j\in\mathbb{Z}}\right\|_{\ell^r(\mathbb{Z})}<\infty.
    $$
\end{definition}
\begin{definition}\cite{bahouriFourierAnalysisNonlinear2011}\label{Bony decomposition}
Considering $u,v\in\mathcal{S}'$, we have the following Bony decomposition:
$$uv=u\prec v+v\prec u+u\circ v,$$
where
$$
u\prec v\triangleq \sum_j S_{j-1}u\Delta_j v\quad \text{and}\quad u\circ v\triangleq\sum_{|k-j|\leq1}\Delta_ku\Delta_jv
$$ 
are called the nonhomogeneous para-product of $v$ by $u$ and the nonhomogeneous remainder of $u$ and $v$, respectively.
\end{definition}
\par We now introduce some function spaces that will be useful in our analysis. 
For convenience, denote
\[
    \ml C^\alpha:=B_{\infty,\infty}^\alpha,\qquad
    H^\alpha:=B_{2,2}^\alpha,\qquad
    W^\alpha:=\ml C^\alpha\cap H^\alpha.
\]
We introduce the space $C_{[s,t]}X$ equipped with the norm
\begin{equation}
    \Vert f\Vert_{C_{[s,t]}X}=\sup_{s\leq\tau\leq t}\Vert f(\tau)\Vert_X.
\end{equation}
For $\beta\in(0,1)$, we also define the space $C^\beta_{[s,t]}X$ endowed with the seminorm
\begin{equation}
    \Vert f\Vert_{C^\beta_{[s,t]}X}
    :=\sup_{s\leq r<\tau\leq t}
    \frac{\Vert f(r)-f(\tau)\Vert_X}{\vert r-\tau\vert^\beta}.
\end{equation}
In particular, we denote
\begin{equation}
    \Vert f\Vert_{C_TX}=\Vert f\Vert_{C_{[0,T]}X},\quad 
    \Vert f\Vert_{C^\beta_T X}=\Vert f\Vert_{C^\beta_{[0,T]} X}.
\end{equation}
Moreover, define
\[
    \widetilde{W}^\alpha([s,t])
    =\widetilde{W}_2^\alpha([s,t])\cap\widetilde{W}_\infty^\alpha([s,t]),
    \qquad
    \widetilde{H}^\alpha([s,t])=\widetilde{W}_2^\alpha([s,t]),
\]
where
\[
    \widetilde{W}_p^\alpha([s,t])
    =C_{[s,t]}B_{p,p}^\alpha\cap C_{[s,t]}^{\frac{\alpha}{2}}L^p,
    \qquad \alpha\in(0,2).
\]
Similarly, we write
\begin{align*}
    \widetilde W_p^\alpha(T)=\widetilde{W}_p^\alpha([0,T]).
\end{align*}

For later use, we introduce a modified version of the paraproduct. Let 
$\varphi\in C_c^\infty(\m R,\m R_+)$ be nonnegative, with compact support contained in $\m R_+$ and total mass $1$, and define $Q_i$ for all $i\geq-1$ by
\begin{align*}
    Q_i:C\ml C^\alpha\to C\ml C^\alpha,\quad 
    Q_if(t)=\int_{\m R}2^{2i}\varphi(2^{2i}(t-s))f(s\vee0)\,\mathrm{d}s,
\end{align*}
where $\ml C^\alpha\triangleq B^\alpha_{\infty,\infty}(\m T,\m R)$ for $\alpha\in\m R$. 
We will often apply $Q_i$, as well as other operators on $C\ml C^\alpha$, to functions 
$f\in C_T\ml C^\alpha$. In this case, we extend $f$ from $[0,T]$ to $\m R_+$ by setting 
$f(t)=f(t\wedge T)$. With the help of $Q_i$, we define the modified paraproduct by
\begin{equation}\label{eq;para-product;modify}
    f\pprec g=\sum_i(Q_iS_{i-1}f)\Delta_i g
\end{equation}
for $f,g\in C(\m R_+,\ml D'(\m T))$.

Let $\ml L=\p_t-\Delta$. We have the following lemmas.
\begin{lemma}[\cite{zhang2026wellposednessregularityfractionalrough}]\label{le;est;paraproduct;modifed}
    Let $\gamma,T\in(0,1)$, $\alpha<0$, $\rho\in(0,2)$ satisfy  $0<\alpha+\rho<2$. Let $f\in \widetilde{W}^\gamma_\infty(T)$, $g(0)=0$, $g\in C_T W^{\alpha+2-\gamma}$ and $\ml Lg\in C_T W^\alpha$. Then we have
    \begin{align*}
        \Vert f\pprec g\Vert_{\widetilde{W}_\infty^{\alpha+\rho}(T)}\lesssim  T^{1-\frac{\rho}{2}}\Vert f\Vert_{\widetilde{W}^\gamma_\infty(T)}(\Vert g\Vert_{C_T\ml C^{\alpha+2-\gamma}}+\Vert \ml Lg\Vert_{C_T\ml C^{\alpha}})
    \end{align*} 
    and
    \begin{align*}
        \Vert f\pprec g\Vert_{\widetilde{W}_2^{\alpha+\rho}(T)}\lesssim  T^{1-\frac{\rho}{2}}\Vert f\Vert_{\widetilde{W}^\gamma_\infty(T)}(\Vert g\Vert_{C_TH^{\alpha+2-\gamma}}+\Vert \ml Lg\Vert_{C_TH^{\alpha}}).
    \end{align*} 
\end{lemma}

\begin{lemma}[\cite{zhang2026wellposednessregularityfractionalrough}]\label{le;heat flow;time smoothing}
   Let $0\leq T\leq 1$, $\alpha<0$, $\delta\in[0,2)$ satisfy $0<\alpha+\delta<2$. Then we have
	\begin{align*}
            \Vert f\Vert_{\widetilde{W}_\infty^{\alpha+\delta}(T)}\lesssim  \Vert f(0)\Vert_{W_\infty ^{\alpha+\delta}}+T^{1-\frac{\delta}{2}}\Vert \ml L f\Vert_{C_TW_\infty^{\alpha}}
        \end{align*}
        and 
        \begin{align*}
            \Vert f\Vert_{\widetilde{W}_2^{\alpha+\delta}(T)}\lesssim \Vert f(0)\Vert_{W_2^{\alpha+\delta}}+T^{1-\frac{\delta}{2}}\Vert \ml L f\Vert_{C_TW_2^{\alpha}}.
	\end{align*}
\end{lemma}

Next, we present lemmas concerning bilinear estimates, which will be used later to establish local well-posedness. We recall the bilinear operator $\widetilde B(f,g)$ in and define $B(f,g)$ as
\begin{flalign}
&B(f,g)(t)\triangleq\int^t_0 e^{ (t-s)\Delta}f(s)\p_xg(s)\,\mathrm{d}s, \\
\end{flalign}

\begin{lemma}\label{lemma;B1}
     Let $s=\sigma+\rho.$ For $0\leq \rho<2$, we have 
    \begin{equation}\label{B}
        \|B(f,g)\|_{L^\infty_T(H^s)}\lesssim T^{1-\frac{\rho}{2}}(1+T^{\frac{\rho}{2}})\|fg_x\|_{L^\infty_T(H^\sigma)}. 
    \end{equation}
    For $0\leq \rho<1$, we have
    \begin{equation}\label{wB}
        \|\widetilde{B}(f,g)\|_{L^\infty_T(H^s)}\lesssim T^{\frac{1-\rho}{2}}\|fg\|_{L^\infty_T(H^{\sigma-2})}. 
    \end{equation}
\end{lemma}
\begin{proof}
    For all $t\in[0,T]$, we have 
    \begin{align}
        \|B(f,g)(t)\|_{H^s}&=\Big(\sum\limits_{k\in\mathbb{Z}}(1+|k|^2)^s|\widehat{B(f,g)}(k)|^2\Big)^{\frac{1}{2}}\notag\\
        &=\Big(\sum\limits_{k\in\mathbb{Z}}|\int_0^t(1+|k|^2)^{\frac{s}{2}}e^{- (t-\tau)|k|^2}|\widehat{fg_x}(k)| \,\mathrm{d}\tau|^2\Big)^{\frac{1}{2}}\notag\\ &=\Big(\sum\limits_{k\in\mathbb{Z}}|\int_0^t(1+|k|^2)^{\frac{\sigma+\rho}{2}}e^{- (t-\tau)|k|^2}|\widehat{fg_x}(k)|\,\mathrm{d}\tau|^2\Big)^{\frac{1}{2}}\notag.
    \end{align}
    Noting that for $k\geq 1$,
    \begin{align*}
        (1+|k|^2)^{\frac{\rho}{2}}e^{- (t-\tau)|k|^2}\lesssim e^{- (t-\tau)|k|^2}|k|^\rho\lesssim \frac{1}{ (t-\tau)^{\frac{\rho}{2}}},
    \end{align*}
    we deduce that
    \begin{align}\label{eq;B1}
        \|B(f,g)(t)\|_{H^s}&\lesssim \Big(\sum\limits_{k\in\mathbb{Z}}|\int_0^t\max(1,\frac{1}{( t-\tau)^{\frac{\rho}{2}}})(1+|k|^2)^{\frac{\sigma}{2}}|\widehat{fg_x}(k)| \,\mathrm{d}\tau|^2\Big)^{\frac{1}{2}}\notag\\
        &\lesssim \int_0^t \max(1,\frac{1}{( t-\tau)^{\frac{\rho}{2}}})\|fg_x\|_{H^\sigma(\m T)} \,\mathrm{d}\tau\notag\\
        &\lesssim T^{1-\frac{\rho}{2}}(1+T^{\frac{\rho}{2}})\|fg_x\|_{L^\infty_T(H^\sigma)},
    \end{align}
    with $0\leq \rho<2$. Similarly, we have
    \begin{align}\label{eq;B2}
        \|\widetilde{B}(f,g)(t)\|_{H^s}&=\Big(\sum\limits_{k\in\mathbb{Z}}(1+|k|^2)^s|\widehat{\widetilde{B}(f,g)}(k)|^2\Big)^{\frac{1}{2}}\notag\\
        &=\Big(\sum\limits_{k\in\mathbb{Z}}|\int_0^t(1+|k|^2)^{\frac{s}{2}-1}e^{- (t-\tau)|k|^2}|k||\widehat{fg}(k)| \,\mathrm{d}\tau|^2\Big)^{\frac{1}{2}}\notag\\ &\lesssim\Big(\sum\limits_{k\in\mathbb{Z}}|\int_0^t(1+|k|^2)^{\frac{\sigma}{2}-1}\frac{1}{(t-\tau)^{\frac{\rho+1}{2}}}|\widehat{fg}(k)| \,\mathrm{d}\tau|^2\Big)^{\frac{1}{2}}\notag\\
        &\lesssim \int_0^t \frac{1}{(t-\tau)^{\frac{\rho+1}{2}}}\|fg\|_{H^{\sigma-2}(\m T)} \,\mathrm{d}\tau\notag\\
        &\lesssim T^{\frac{1-\rho}{2}}\|fg\|_{L^\infty_T(H^{\sigma-2})},
    \end{align}
    with $0\leq \rho<1$.\par
    Combining inequalities \eqref{eq;B1} and \eqref{eq;B2} completes the proof.
\end{proof}

Using Lemma \ref{lemma;B1} and the product laws in Besov spaces, we can prove that
\begin{lemma}\label{estimate_B}
    Let $s\in[\max(1,3-\kappa),\kappa)$, where $\kappa$ satisfies
    $$
    \kappa\in\left\{
    \begin{array}{ll}
    (2,2\eta+\frac{3}{2}), & \text{if}\quad \eta>\frac{1}{4}, \\
    (\frac{3}{2},2\eta+\frac{3}{2}), & \text{if}\quad 0<\eta\leq\frac{1}{4}.
    \end{array}\right.
    $$    
    Denote 
    $$
    \rho_1\triangleq\max(s-(\kappa-2),0),\quad \rho_2\triangleq\max(s-(\kappa-1),0).
    $$
    Then for any $T>0$, we have the following estimates:\\
    (1) For $f\in L^{\infty}_T(H^s)$, we have
    \begin{align*}
        &\|B(f,f)\|_{L^{\infty}_T(H^s)}\leq C_1T^{\frac{1}{2}}(1+T^{\frac{1}{2}})\|f\|_{L^{\infty}_T(H^s)}^2,\\
        &\|\widetilde{B}(f,f)\|_{L^{\infty}_T(H^s)}+\|\widetilde{B}(f_x,f_x)\|_{L^{\infty}_T(H^s)}\leq C_1T^{\frac{1}{4}}\|f\|_{L^{\infty}_T(H^s)}^2.
    \end{align*}
    (2) For $f\in L^{\infty}_T(W^{\kappa-1}),\ g\in L^{\infty}_T(W^{2\kappa-1})$, we have
    \begin{align*}
        &\|B(f,f)\|_{L^{\infty}_T(H^s)}\leq C_2T^{1-\frac{\rho_1}{2}}(1+T^{\frac{\rho_1}{2}})\|f\|_{L^{\infty}_T(W^{\kappa-1})}^2,\\
        &\|\widetilde{B}(f,f)\|_{L^{\infty}_T(H^s)}\leq C_2 T^{\frac{1}{2}}\|f\|_{L^{\infty}_T(W^{\kappa-1})}^2,\\
        &\|\widetilde{B}(g_x,g_x)\|_{L^{\infty}_T(H^s)}\leq C_2 T^{\frac{1}{2}}\|g\|_{L^{\infty}_T(W^{2\kappa-1})}^2,\\
        &\|\widetilde{B}(f_x,g_x)\|_{L^{\infty}_T(H^s)}\leq C_2 T^{\frac{1}{2}}\|f\|_{L^{\infty}_T(W^{\kappa-1})}\|g\|_{L^{\infty}_T(W^{2\kappa-1})}.
    \end{align*}
    Furthermore, in the case $\kappa\in(2,2\eta+\frac{3}{2})$ and $\eta>\frac{1}{4}$, it holds that
    \begin{align*}
        \|\widetilde{B}(f,f)\|_{L^{\infty}_T(H^s)}+\|\widetilde{B}(f_x,f_x)\|_{L^{\infty}_T(H^s)}\leq C_2T^{\frac{1-\rho_2}{2}}\|f\|_{L^{\infty}_T(W^{\kappa-1})}^2.
    \end{align*}
    (3) For $f\in L^{\infty}_T(W^{\kappa-1}),\ g\in L^{\infty}_T(H^s)$, we have
    \begin{align*}
        &\|B(f,g)\|_{L^{\infty}_T(H^s)}\leq C_3T^{\frac{1}{2}}(1+T^{\frac{1}{2}})\|f\|_{L^{\infty}_T(W^{\kappa-1})}\|g\|_{L^{\infty}_T(H^s)},\\
        &\|B(g,f)\|_{L^{\infty}_T(H^s)}\leq C_3T^{1-\frac{\rho_1}{2}}(1+T^{\frac{\rho_1}{2}})\|f\|_{L^{\infty}_T(W^{\kappa-1})}\|g\|_{L^{\infty}_T(H^s)},\\
        &\|\widetilde{B}(f,g)\|_{L^{\infty}_T(H^s)}+\|\widetilde{B}(f_x,g_x)\|_{L^{\infty}_T(H^s)}\leq C_3 T^{\frac{1}{4}}\|f\|_{L^{\infty}_T(W^{\kappa-1})}\|g\|_{L^{\infty}_T(H^s)}.
    \end{align*}
\end{lemma}

  \section{Enhanced stochastic process}\label{sec;enhanced data}
 \par The regularity of the stochastic convolution plays a decisive role in the
analysis of the nonlinear term. Depending on the value of $\eta$, different
singular structures emerge and require different analytical treatments. We
therefore divide the argument into four regimes $\eta\in(1/4,\infty)$, $\eta\in(1/8,1/4]$, $\eta\in(0,1/8]$, $\eta\in(-1/12,0]$. We begin by recalling the stochastic objects that encode the rough behavior
of the solution. For space-time white noise $\xi$ and $X(t):=\int_0^t e^{ (t-s)\Delta} (1-\p^2_x)^{-\eta}\xi(s)\,\mathrm{d}s$, we write 
\begin{equation}\label{eq;X(t) decomposition}
    X(t)=\int_{0}^t e^{ (t-s)\Delta} (1-\p^2_x)^{-\eta}\xi(s)\,\mathrm{d}s.
\end{equation}
 For $X(t)$, we have the following lemma.
\begin{lemma}[\cite{zhang2026wellposednessregularityfractionalrough}]\label{X_bdd}
    The spatial Fourier transform $X_k=\mathcal{F}_x X(k)$ is a Gaussian process with zero-mean satisfying
    \begin{equation*}
        \mathbb{E}[X_k(t)X_{k'}(t')]=(e^{-\vert t-t'\vert\vert k\vert^\gamma}-e^{-( t+t')\vert k\vert^\gamma})\frac{1}{2\vert k\vert^\gamma|\m T|}\textbf{1}_{k+k'=0}(1+|k|^2)^{-2\eta},
    \end{equation*}
    and $\m E[X_0(t)X_{k}(t')]= \frac{1}{|\m T|}(t\wedge t')\textbf{1}_{k=0}$.
    where $|\mathbb{T}|=\int_\mathbb{T}1\,\mathrm{d}x$ and $\mathbf{1}_{k+k^{\prime}=0}$ is the indicator function. Moreover, for any $T>0$ and $\kappa<2\eta-2\delta+\frac32$, we have $X\in C_T^\delta W^{\kappa-1}$ for any $0<\delta<\frac12.$
\end{lemma}
When $\eta>1/4$, the stochastic convolution is sufficiently regular, namely
\[
X\in C_TW^{1+}.
\]
Therefore, the nonlinear interaction does not require any further cancellation, since $(\partial_xX)^2$ is well-defined in the classical sense.
We next turn to the intermediate regime \(\eta\in(1/8,1/4]\), where the regularity of \(X\) is no longer sufficient to define the nonlinear term by classical multiplication. More precisely, after applying Bony's decomposition, the only obstructing component is the resonant interaction
\[
\partial_xX\circ\partial_xX .
\]
The convergence of this term relies on a cancellation mechanism inherent to
the Gaussian structure of the stochastic convolution.

\begin{lemma}\label{le;resonant}
    	Let $\eta\in (1/8,1/4]$, and $X_\sigma$ satisfy
    	\begin{equation}\label{eq;X_sigma}
    		\p_tX_\sigma- \Delta X_\sigma=(1-\p_x^2)^{-\eta}\xi_\sigma,\quad X_\sigma(0)=0,
    	\end{equation}
    	then for $\eta\in(\frac{1}{8},\frac{1}{4}]$, $\kappa<2\eta+\frac{3}{2}$, we have 
    	\begin{align}\label{eq;convergence;f_xf_x}
    		\p_x(\p_xX_\sigma\p_xX_\sigma)\to& \p_x(\p_xX\p_x X)
    	\end{align}
    	in $C_TW^{2\kappa-5}$ almost surely as $\sigma\to0$. Moreover, we have the estimate
        \begin{align*}
            \Vert P(D)(\p_x X\cdot\p_x X)\Vert_{C_TW^{2\kappa-3}}\lesssim  \Vert X\Vert^2_{C_TW^{\kappa-1}}+\Vert \p_x(\p_x X\circ\p_x X)\Vert_{C_TW^{2\kappa-5}}
        \end{align*}
    \end{lemma}
    \begin{proof}
    	By Bony's decomposition, we have $$\p_x(\p_xX_\sigma\p_xX_\sigma)=2\p_x(\p_xX_\sigma\prec\p_xX_\sigma)+\p_x(\p_xX_\sigma\circ\p_xX_\sigma).$$
    	The paraproduct $\prec$ is stable under this convergence. Indeed, for $\kappa<2\eta+\frac{3}{2}$, since $X_\sigma$ converges to $X$ in $C_TW^{\kappa-1}$ by Lemma \ref{X_bdd}, the  continuous of $\prec$ yields that 
        \begin{align*}
            \p_x(\p_xX_\sigma\prec\p_xX_\sigma)\to\p_x(\p_xX\prec\p_xX)
        \end{align*}
           	in $C_TW^{2\kappa-5}$. Thus, it remains to prove the convergence of 
    	$\p_x(\p_xX_\sigma\circ\p_xX_\sigma)$ in $C_TW^{2\kappa-5}$. This follows from the fact that $\p_x(\p_xX\circ\p_xX)\in C_TW^{4\eta-2-}$, which can be established using the dominated convergence theorem. Note that
    	\begin{equation}
    		\Delta_l(\p_x(\p_xX\circ\p_xX))=\sum_{\vert i-j\vert \leq 1}1_{i\gtrsim l}\Delta_l\p_x (\Delta_i(\p_xX)\Delta_j(\p_xX))
    	\end{equation}
    	and that $\ml F(\Delta_if\Delta_jg)$ is supported on a ball of radius $c2^{i+j}$. To avoid case distinctions, we only work with $l>N$, so that we can always assume $i,j>0$. The case $l<N$ can be handled using essentially the same arguments.
    	
    	Let $\widehat X_{s,t}(k)=\widehat X(t,k)-\widehat X(s,k)$. We firstly consider $\m E[\widehat{X}(t,k)\widehat{X}_{s,t}(k')\overline{\widehat{X}(t,l)}\overline{\widehat{X}_{s,t}(l')}]$, where $k+k'\neq0$, $0\leq s<t\leq T$. By Wick's theorem, we have 
    	\begin{align*}
    		\m E[\widehat{X}(t,k)\widehat{X}_{s,t}(k')\overline{\widehat{X}(t,l)}\overline{\widehat{X}_{s,t}(l')}]=&\m E[\widehat{X}(t,k)\widehat{X}_{s,t}(k')]\m E[\overline{\widehat{X}(t,l)}\overline{\widehat{X}_{s,t}(l')}]\\
    		&+\m E[\widehat{X}(t,k)\overline{\widehat{X}_{s,t}(l')}]\m E[\widehat{X}_{s,t}(k')\overline{\widehat{X}(t,l)}]\\
    		&+\m E[\widehat{X}(t,k)\overline{\widehat{X}(t,l)}]\m E[\widehat{X}_{s,t}(k')\overline{\widehat{X}_{s,t}(l')}]\\
    		=&I_1+I_2+I_3.
    	\end{align*}
    	$I_1=0$ since $k+k'\neq0$. For $\delta\in[0,1]$, we deduce from Lemma \ref{X_bdd} that
    	\begin{align*}
    		I_2=&\m E[\widehat{X}(t,k)\overline{\widehat{X}_{s,t}(l')}]\m E[\widehat{X}_{s,t}(k')\overline{\widehat{X}(t,l)}]\\
    		=&\frac{1}{4|\m T|^2} 1_{k=l'}1_{k'=l}\frac{(1-e^{-(t-s)|k|^2})(1+e^{-(s+t)|k|^2})(1-e^{-(t-s)|k'|^2})(1+e^{-(s+t)|k'|^2})}{|k|^2 |k'|^2(1+|k|^2)^{2\eta}(1+|k'|^2)^{2\eta}}\\
    		\lesssim& |t-s|^\delta |k|^{2\delta-2-4\eta}|k'|^{-2-4\eta}
    	\end{align*}
    	and 
    	\begin{align*}
    		I_3=&\m E[\widehat{X}(t,k)\overline{\widehat{X}(t,l)}]\m E[\widehat{X}_{s,t}(k')\overline{\widehat{X}_{s,t}(l')}]\\
    		=&\frac{1}{4|\m T|^2}I_{k'=l'}I_{k=l}\frac{(2-e^{-2s|k'|^2}-e^{-2t|k'|^2}-2e^{-(t-s)|k'|^2}+2e^{-(t+s)|k'|^2})(1-e^{-2t|k|^2})}{|k'|^2|k|^2(1+|k|^2)^{2\eta}(1+|k'|^2)^{2\eta}}\\
    		\lesssim&|t-s|^\delta |k|^{-2-4\eta}|k'|^{2\delta-2-4\eta}.
    	\end{align*}
    	Therefore, we obtain 
    	\begin{align*}
    		&\m E[\vert \sum_{\vert i-j\vert \leq 1}1_{i\gtrsim l}\Delta_l(\p_x (\Delta_i(\p_xX(t))\Delta_j(\p_x(X)_{s,t})))\vert^2]\\
    		\lesssim&\m E[\vert \sum_{k,k'\in \m Z}\sum_{\vert i-j\vert \leq 1}1_{i\gtrsim l}\rho_l(k+k')\rho_i(k)\rho_j(k')|k+k'|kk'\widehat{X}(t,k)\widehat{X}_{s,t}(k')e^{i(k+k')x}\vert^2]\\
    		\lesssim&\sum_{k,k'\in \m Z}\sum_{\vert i-j\vert \leq 1}\sum_{\vert i'-j'\vert \leq 1}1_{i\gtrsim l}1_{i'\gtrsim l}\rho^2_l(k+k')\rho_i(k)\rho_j(k')(\rho_{i'}(k)\rho_{j'}(k')+\rho_{i'}(k')\rho_{j'}(k))\\
    		& |k+k'|^2|k|^2|k'|^2|t-s|^\delta |k|^{-2-4\eta}|k'|^{-2-4\eta}(|k|^{2\delta}+|k'|^{2\delta})\\
    		\lesssim&\sum_{i\gtrsim l}2^{3l}|t-s|^\delta2^{i(1+2\delta-8\eta)}
    	\end{align*}
    	for all $\delta\in[0,1]$. Since $\eta>\frac{1}{8}$, we get 
    	\begin{align*}
    		\m E[\vert \sum_{\vert i-j\vert \leq 1}&1_{i\gtrsim l}\Delta_l (\p_x(\Delta_i(\p_xX(t))\Delta_j(\p_x(X)_{s,t})))\vert^2]\lesssim|t-s|^\delta2^{l(4+2\delta-8\eta)}
    	\end{align*}
    	for $\delta\in[0,1]$ small enough, and
    	\begin{align*}
    		\m E[\vert \sum_{\vert i-j\vert \leq 1}&1_{i\gtrsim l}\Delta_l (\p_x(\Delta_i(\p_x(X)_{s,t})\Delta_j(\p_xX(s)))\vert^2]\lesssim|t-s|^\delta2^{l(4+2\delta-8\eta)}
    	\end{align*}
    	also holds by the same argument. Similar to the proof of Lemma 4.6 in \cite{gubinelli2015paracontrolled}, we obtain 
    	\begin{align*}
    		\m E[\Vert \Delta_l(\p_x(\p_xX\circ\p_xX))_{s,t}\Vert_{L^2}^{2p}]\lesssim |t-s|^{\delta p}2^{pl(4+2\delta-8\eta)}.
    	\end{align*}
    	Thus, since $\kappa<2\eta+\frac{3}{2}$, we have
    	\begin{align*}
    		\m E[\Vert (\p_x(\p_xX\circ\p_xX))_{s,t}\Vert_{H^{2\kappa-5}}^{2p}]\lesssim&(\sum_{l\geq -1}2^{2l({2\kappa-5})}(\m E[\Vert \Delta_l(\p_x(\p_xX\circ\p_xX))_{s,t}\Vert_{L^2}^{2p}])^{\frac{1}{p}})^p\\
    		\lesssim&(\sum_{l\geq -1}2^{2l(2\kappa-3-4\eta+\delta)})^p|t-s|^{\delta p}\lesssim |t-s|^{\delta p}
    	\end{align*}
    	for $\delta\in[0,1]$ small enough. For the H\"older space, a similar argument gives
    	\begin{align*}
    		\m E[\Vert (\p_x(\p_xX\circ\p_xX))_{s,t}\Vert_{B^{2\kappa-5}_{2p,2p}}^{2p}]\lesssim&\sum_{l\geq -1}2^{2l({2\kappa-5})p}\Vert\m E |\Delta_l(\p_x(\p_xX\circ\p_xX))_{s,t})|^2\Vert_{L^p}^{p}\\
    		\lesssim&\sum_{l\geq -1}2^{2l(2\kappa-3-4\eta+\delta)}|t-s|^{\delta p}\lesssim |t-s|^{\delta p}.
    	\end{align*}
    	Choosing $p$ large enough, we deduce from Kolmogorov's theorem that $\p_x(\p_xX\circ \p_xX)\in C_TW^{2\kappa-5}$.
        \par Finally, since $\kappa<2\eta+\frac32\leq2$, we have $\kappa-2<0$.
Hence, by the paraproduct estimate and the definition
$W^\alpha=\mathcal C^\alpha\cap H^\alpha$,
\[
    \|\partial_x(\partial_xX\prec\partial_xX)\|_
    {C_TW^{2\kappa-5}}
    \lesssim
    \|\partial_xX\|_{C_TW^{\kappa-2}}^2
    \lesssim
    \|X\|_{C_TW^{\kappa-1}}^2.
\]

Since $ P(D)=(1-\partial_x^2)^{-1}\partial_x$ and $(1-\partial_x^2)^{-1}$ gains two spatial derivatives, it follows that
\begin{align*}
    \|P(D)((\partial_xX)^2)\|_{C_TW^{2\kappa-3}}
    &\lesssim
    \|\partial_x((\partial_xX)^2)\|_{C_TW^{2\kappa-5}}\\
    &\lesssim
    \|X\|_{C_TW^{\kappa-1}}^2
    +
    \|\partial_x(\partial_xX\circ\partial_xX)\|_
    {C_TW^{2\kappa-5}}.
\end{align*}
This proves the desired estimate.
    \end{proof}
 
     \par For the rough regime $\eta\in(0,1/8]$, the resonant product
\[
\partial_xX\circ\partial_xX
\]
cannot be constructed by a direct limiting procedure. Instead, we exploit the smoothing effect of the heat semi-group and introduce a higher-order stochastic expansion. Recall the stationary process $Y$ satisfying \eqref{eq;stationary} and let $Y_\sigma$ be the stationary solution of 
     \begin{equation}
       \p_t Y_\sigma-\Delta Y_\sigma=\Pi_0^{\perp}(1-\p_x^2)^{-\eta}\xi_\sigma,\quad Y_\sigma(0)=\int_{-\infty}^0 e^{-s\Delta}(1-\p_x^2)^{-\eta}\Pi_0^{\perp}\xi_\sigma(s)\,\mathrm ds
     \end{equation}
   We define $X_\sigma ^{\dag 2}$ as the second-order expansion of $(\p_x Y_\sigma\p_x Y_\sigma)$ under the heat flow $S(t)=e^{t\Delta}$, namely,
   \begin{equation}
       \p_t X_\sigma^{\dag 2}-\Delta X_\sigma^{\dag 2}=P(D)(\p_x Y_\sigma)^2,\quad X_{\sigma}^{\dag 2}(0)=0
   \end{equation}
   Our aim is to show that  $X_\sigma ^{\dag2}$ converge to a well-defined limiting process in $C_TW^{1+}$. For simplicity, we first introduce an auxiliary process $Y_\sigma^{\dag 2}$, defined as the the stationary solution of
   \begin{equation}\label{eq; Ydag2}
   	\partial_t Y_\sigma^{\dag 2}-\Delta Y_\sigma^{\dag 2}=\Pi_0^\perp(\p_x Y_\sigma)^2,\quad Y_\sigma^{\dag 2}(0)=\int_{-\infty}^0e^{-s\Delta}\Pi_0^\perp(\p_x Y_\sigma)^2\,\mathrm{d}s.
   \end{equation}
   the relation
   \begin{equation}
   	X_\sigma^{\dag 2}=P(D)(Y_\sigma^{\dag 2}-S(t)Y_\sigma^{\dag 2}(0))
   \end{equation}
   implies that it suffice to show the convergence of $Y_\sigma^{\dag2}$. Indeed, once the convergence of $Y_\sigma^{\dag 2}$ is established, the convergence of $X_\sigma^{\dag 2}$ follows from the continuity of $P(D)$ together with the smoothing properties of the heat semigroup.
   \par  For $Y_\sigma^{\dag 2}$ and $X_\sigma^{\dag 2}$, we have the following lemma.
    \begin{lemma}
    	Let $\eta \in(-1/8,1/8)$, then for any $T\geq0$, $\alpha<1+4\eta$ and $\beta<\min(1/2+2\eta,1/2)$, there exists $Y^{\dag 2}\in C_TW^\alpha\cap C_T^\beta W^{\alpha-2\beta}$ a.s., such that $Y^{\dag 2}_\sigma\to Y^{\dag 2}$ in $L^2(\Omega;C_TW^\alpha\cap C_T^\beta W^{\alpha-2\beta})$ as $\sigma\to0$. Moreover, for $\kappa<2\eta+\frac{3}{2}$, there exists $X^{\dag 2}\in C_TW^{2\kappa-1}\cap C_T^\beta W^{2\kappa-1-2\beta}$ a.s., such that   $X_\sigma^{\dag 2}\to X^{\dag 2}$ in $L^2(\Omega;C_TW^{2\kappa-1}\cap C_T^\beta W^{2\kappa-1-2\beta})$.
        \begin{proof}
    Denote $\m Z_*=\m Z\backslash\{0\}$, $L=(l,m)\in\m Z_*^2$, $L'=(l',m')\in\m Z_*^2$, $C_\sigma(L)=\phi(\sigma l)\phi(\sigma m)$ and $\psi(L)=l+m$. The Fourier mode of $Y_\sigma^{\dag 2}$ can be represented as 
        \begin{align*}
            Y_\sigma^{\dag 2}=\sum_{L\in \m Z_*^2,\psi(L)\neq0} C_\sigma(L) f_L(t)e^{i\psi(L)x},
        \end{align*}
        where $f_L(t)=\int_{-\infty}^t e^{-|\psi(L)|^2(t-s)}|l||m|\widehat Y(s,l)\widehat Y(s,m)\mathrm ds$. Let $F_{L,L'}(t,t')=\m E[f_{L}(t)f_{L'}(t')]$, $\widehat F_{L,L'}(t,t')=\m E[(f_L(t)-f_L(t'))(f_{L'}(t)-f_{L'}(t'))]$. Proposition A.1 in \cite{zhang2026generalizedframeworksingularfractional} yields that it suffices to establish the estimates
             \begin{align*}
                \begin{gathered}
                \sum_{\substack{L,L'\in\m Z_*^2\\\psi(L)\neq0,\psi(L')\neq0}}\sup_{t\in[0,T]}|\psi(L)|^{1+4\eta}|\psi(L')|^{1+4\eta}F_{L,L'}(t)<\infty\\
                \text{and}
                \sum_{\substack{L,L'\in\m Z_*^2\\\psi(L)\neq0,\psi(L')\neq0}}|\psi(L)|^{1+4\eta-\beta\gamma}|\psi(L')|^{1+4\eta-\beta\gamma}\sup_{s,t\in[0,T]}\frac{\widehat F_{L,L'}(t)}{|t-s|^{2\beta}}<\infty.
            \end{gathered}
            \end{align*}
       Recalling the proof of the proposition in \cite{hairer2013solving}, we have
            \begin{align*}
            \begin{gathered}
                F_{L,L'}(t,t)=\Delta_{L,L'}\frac{1}{2(l+m)^2(((l+m)^2-lm)|l|^{4\eta}|m|^{4\eta}},\\
                \widehat F_{L,L'}(t,t')=\Delta_{L,L'}\frac{(l^2+m^2)^\beta |l+m|^{2\beta}|t-t'|^{2\beta}}{(l+m)^2((l+m)^2-lm)|l|^{4\eta}|m|^{4\eta}}
            \end{gathered}
            \end{align*}
where $\beta\in[0,1/2]$ and $\Delta_{L,L'}=\textbf{1}_{\psi(L)+\psi(L')=0}(\textbf{1}_{l+l'=0}+\textbf{1}_{m+l'=0})$. It's not difficult to verify that the above series are summable by Lemma 5.1 in \cite{zhang2026generalizedframeworksingularfractional}.
        \end{proof}
    \end{lemma}

    \par For the case of $\eta\in (-1/12,0]$, the singularity does not only arise from the resonant interaction of the stochastic convolution itself, but also from the interaction between the remainder term $v $ and $Y$ . To overcome this difficulty, we employ the paracontrolled approach, which provides a powerful framework for studying a wide class of singular PDEs. We first recall the definition of a paracontrolled distribution. 
\begin{definition}
    Let $\kappa>1$. A distribution $u\in H^{\kappa}$ is said to be 
    paracontrolled by $\ml Q\in H^\kappa$ if it admits a decomposition
    \[
        u=u'\prec\ml Q+u^\sharp,
    \]
    for some $u'\in\ml C^{\kappa-1}$, where the remainder $u^\sharp$ satisfies
    \[
        u^\sharp\in H^{2\kappa-1}.
    \]
\end{definition}
Let $Y\in W^{\kappa-1}$ define
\begin{equation}
    \mathcal Q(t)=\int_0^t e^{(t-s)\Delta}\partial_xY(s)\,\mathrm{d}s.
\end{equation}
The paracontrolled ansatz takes the form
$$
v=u'\prec\mathcal Q+u^\sharp,
$$
where \(u'\prec\mathcal Q\) has essentially the same regularity as \(\mathcal Q\), whereas the remainder \(u^\sharp\) enjoys higher regularity. Thus, the leading singular behavior of \(v\in H^{\kappa}\) is captured by the paraproduct \(u'\prec\mathcal Q\). Standard heat regularity estimates alone do not suffice to give a meaning to the product
$$
\partial_xv\,\partial_xY
$$
when $\eta\leq 0$ since $\kappa<2\eta+\frac{3}{2}$. Since \(u^\sharp\) is more regular than \(v\), the main difficulty is concentrated in the resonant product
$$
\partial_x(u'\prec\mathcal Q)\circ\partial_xY.
$$

Paracontrolled calculus provides a convenient way to handle this singular term. The key observation is that \(\partial_x(u'\prec\mathcal Q)\circ\partial_xY\) can be given a meaning once the resonant product
$$
\partial_x\mathcal Q\circ\partial_xY
$$
is suitably defined. Indeed, we decompose
\begin{align*}
\partial_x(u'\prec\mathcal Q)\circ\partial_xY
=&\big[\partial_x(u'\prec\mathcal Q)-u'\prec\partial_x\mathcal Q\big]\circ\partial_xY
\\&+\big[(u'\prec\partial_x\mathcal Q)\circ\partial_xY
-u'(\partial_x\mathcal Q\circ\partial_xY)\big]
+u'(\partial_x\mathcal Q\circ\partial_xY).
\end{align*}
By Lemma~\ref{le;commutater;D}, the commutator
$$
\partial_x(u'\prec\mathcal Q)-u'\prec\partial_x\mathcal Q
$$
gains $\kappa-1$-orders regularity, and hence the first term on the right-hand side is well-defined. Although the second term does not arise from such a derivative gain, it can nevertheless be controlled by Lemma \ref{le;commutator;resonant}. More precisely
\begin{align*}
    C(u',\p_x \ml Q,\p_x Y)=\big[(u'\prec\partial_x\mathcal Q)\circ\partial_xY
-u'(\partial_x\mathcal Q\circ\partial_xY)\big]
\end{align*}
defines a bounded operator from $\ml C^{\kappa-1}\times W^{\kappa-1}\times W^{\kappa-2}$ to $W^{3\kappa-4}$ provided that $3\kappa-4>0$. Since $\kappa<2\eta+\frac{3}{2}$, such a $\kappa$ can be chosen whenever $2\eta+\frac{3}{2}>\frac{4}{3}$, namely, when $\eta>-\frac{1}{12}
$.
in Lemma~\ref{le;commutator;resonant}. Consequently, the only genuinely singular object that remains to be defined is
$$
\partial_x\mathcal Q\circ\partial_xY.
$$
Once this resonant product is constructed with the required regularity, the singular term above, and hence \(\partial_xv\,\partial_xY\), can be interpreted within the paracontrolled framework. We establish the required construction in the following lemma.
\begin{lemma}\label{le;high order}
        For $\eta\in(-\frac18,\frac{1}{8}]$, $\kappa<2\eta+\frac{3}{2}$, we have
        \begin{equation}
            \p_x\ml Q_\sigma\circ \p_xY_\sigma\to \p_x \ml Q\circ \p_xY
        \end{equation}
        in $L^2(\Omega;C_TW^{2\kappa-3})$ as $\sigma\to 0$.
        \begin{proof}
          Using a similar argument as in the proof of Lemma~\ref{le;resonant}, it suffices to estimate $\m E[\Vert(\p_x\ml Q\circ\p_xY)_{s,t}\Vert_{H^{2\kappa-3}}^{2p}]$ and $\m E[\Vert (\p_x\ml Q\circ\p_xY)_{s,t}\Vert_{B^{2\kappa-3}_{2p,2p}}^{2p}]$. Let $\widehat {\ml Q}_{s,t}(k)=\widehat {\ml Q}(t,k)-\widehat {\ml Q}(s,k)$ and $\widehat {Y}_{s,t}(k)=\widehat {Y}(t,k)-\widehat {Y}(s,k)$. Notice that $\widehat {\ml Q}(t,k)\widehat Y(t,k')-\widehat {\ml Q}(s,k)\widehat Y(s,k')=\widehat {\ml Q}_{s,t}(k)\widehat Y(t,k')+\widehat {\ml Q}(s,k)\widehat Y_{s,t}(k')$. Therefore, we have
          \begin{align*}
              \m E[(\widehat {\ml Q}(t,k)\widehat Y(t,k')-&\widehat {\ml Q}(s,k)\widehat Y(s,k'))((\widehat {\ml Q}(t,l)\widehat Y(t,l')-\widehat {\ml Q}(s,l)\widehat Y(s,l')))]\\
              =& \m E[(\widehat {\ml Q}_{s,t}(k)\widehat Y(t,k')+\widehat {\ml Q}(s,k)\widehat Y_{s,t}(k'))((\widehat {\ml Q}_{s,t}(l)\widehat Y(t,l')+\widehat {\ml Q}(s,l)\widehat Y_{s,t}(l')))]\\
              =&\m E[\widehat {\ml Q}_{s,t}(k)\widehat Y(t,k')\widehat {\ml Q}_{s,t}(l)\widehat Y(t,l')]+\m E[\widehat {\ml Q}(s,k)\widehat Y_{s,t}(k')\widehat {\ml Q}_{s,t}(l)\widehat Y(t,l')]\\
              &+\m E[\widehat {\ml Q}_{s,t}(k)\widehat Y(t,k')\widehat {\ml Q}(s,l)\widehat Y_{s,t}(l')]+\m E[\widehat {\ml Q}(s,k)\widehat Y_{s,t}(k')\widehat {\ml Q}(s,l)\widehat Y_{s,t}(l')].
          \end{align*}
        By Wick's theorem and Lemma~\ref{le;covariance; QX}, we obtain
        \begin{align*}
               \m E[(\widehat {\ml Q}(t,k)\widehat Y(t,k')-&\widehat {\ml Q}(s,k)\widehat Y(s,k'))((\widehat {\ml Q}(t,l)\widehat Y(t,l')-\widehat {\ml Q}(s,l)\widehat Y(s,l')))]\lesssim \frac{(t-s)^\delta |k|^{2\delta}}{(1+|k|^2)^{4\eta}|k|^6}.
        \end{align*}
        At the same time, using a similar argument as in Lemma~\ref{le;resonant}, we have
        \begin{align*}
    		\m E[\vert \sum_{\vert i-j\vert \leq 1}1_{i\gtrsim l}\Delta_l(\Delta_i\p_x\ml Q\Delta_j\p_xY)_{s,t}\vert^2]\lesssim\sum_{i\gtrsim l}2^{l}|t-s|^\delta2^{i(2\delta-8\eta-1)}.
    	\end{align*}
      Since $\eta>-1/8$, we may choose $\delta>0$ sufficiently small such that $2\delta-8\eta-1<0$. Therefore, it follows from Gaussian hypercontractivity that 
      \begin{align*}
    		\m E[\Vert \Delta_l((\p_x\ml Q\circ\p_xY))_{s,t}\Vert_{L^2}^{2p}]\lesssim |t-s|^{\delta p}2^{pl(2\delta-8\eta)}.
    	\end{align*}
 By Minkowski's inequality, for $\kappa<2\eta+\frac{3}{2}$, we may choose $\delta>0$ sufficiently small such that $2\kappa-3+\delta-4\eta<0$, and obtain
    	\begin{align*}
    		\m E[\Vert(\p_x\ml Q\circ\p_xY)_{s,t}\Vert_{H^{2\kappa-3}}^{2p}]=&\m E[(\sum_{l\geq-1}2^{2l(2\kappa-3)}\Vert \Delta_l(\p_x\ml Q\circ\p_xY)_{s,t}\Vert_{L^2}^2)^p]\\
            \leq&(\sum_{l\geq -1}2^{2l(2\kappa-3)}(\m E[\Vert \Delta_l(\p_x\ml Q\circ\p_xY)_{s,t}\Vert_{L^2}^{2p}])^{\frac{1}{p}})^p\\
    		\lesssim&(\sum_{l\geq -1}2^{2l(2\kappa-3+\delta-4\eta)})^p|t-s|^{\delta p}\lesssim |t-s|^{\delta p}
    	\end{align*}
  and 
      	\begin{align*}
    		\m E[\Vert (\p_x\ml Q\circ\p_xY)_{s,t}\Vert_{B^{2\kappa-3}_{2p,2p}}^{2p}]
            \lesssim&
            \sum_{l\geq -1}2^{2l(2\kappa-3) p}\Vert\m E |\Delta_l(\p_x\ml Q\circ\p_xY)_{s,t}|^2\Vert_{L^p}^{p}\\
    		\lesssim&
            \sum_{l\geq -1}2^{2l(2\kappa-3+\delta-4\eta)p}|t-s|^{\delta p}\lesssim |t-s|^{\delta p}.
    	\end{align*}
      	Choosing $p$ large enough, we deduce from Kolmogorov's theorem that $\p_x\ml Q\circ \p_xY\in C_TW^{2\kappa-3}$. The convergence follows by the same argument as in the proof of Lemma~\ref{le;resonant}. This completes the proof. 
        \end{proof}
     \end{lemma}

\section{Local well-posedness}\label{sec;wellposedness}
In this section, we study the local well-posedness of \eqref{eq;CH;diffusion}.
\subsection{The case \texorpdfstring{$\eta\in(1/4,\infty)$}{eta in (1/4, infinity)}}
We begin with the least singular regime, in which the stochastic convolution \(X\) is sufficiently regular for the nonlinear interactions to be interpreted without introducing any additional stochastic enhancement. Recall that \(X\) is defined by \eqref{eq;X(t) decomposition}. Applying the Da Prato--Debussche decomposition
$$
    u=X+v,
$$
we obtain the following equation for the remainder \(v\):
\begin{equation}\label{v}
\partial_t v-\Delta v
=-(v+X)\partial_x(v+X)
+P(D)\left((v+X)^2+\frac12(v_x+X_x)^2\right).
\end{equation}
In this regime, the regularity of \(X\) is sufficient to give a classical meaning to all the products appearing in \eqref{v}. We can therefore establish local well-posedness by applying standard deterministic estimates to the remainder equation. The precise statement is given in the following proposition.

\begin{proposition}\label{local_thm}
    Let $\eta>\frac{1}{4},\ \kappa\in(2,2\eta+\frac{3}{2}),\ s\in[1,\kappa)$ and $u_0\in H^s(\m T)$. Then there exist a positive time $T^*=C(\Vert u_0\Vert_{H^s(\m T)},\Vert X\Vert_{C_1W^{\kappa-1}(\m T)})$ and a solution $u(t,x)$ of equation \eqref{eq;CH;diffusion} satisfying
    $$u(t,x)=X(t,x)+v(t,x),\quad t\in [0,T^*],$$
    where $v(t,x)$ satisfies \eqref{v} with $v(t,x)\in C_{T^*}H^s(\m T)~ a.s.$ and $X(t,x)$ satisfies \eqref{eq;X(t) decomposition} with $X(t,x)\in C_{T^*}W^{\kappa-1}(\m T)~ a.s.$.      
\end{proposition}
\begin{proof}
    By Lemma \ref{X_bdd}, we have $X\in C_{1}(W^{\kappa-1})$ where $X$ satisfies \eqref{eq;X(t) decomposition}. Hence, it suffices to prove that there exists a positive time $T^*\leq 1$ such that $v\in C_{T^*}(H^s)$, where $v$ satisfies \eqref{v}. 
    
    Let $r\triangleq2\|u_0\|_{H^s}$. Denote
    \begin{equation}\label{def;B}
         \mathcal{B}(f,g)\triangleq B(f,g)+B(g,f)+2\widetilde{B}(f,g)+\widetilde{B}(f_x,g_x).
    \end{equation}
    The solution $v$ of \eqref{v} takes the mild form of
    \begin{equation}
          v(t,x)=P(t)u_0-[\frac{1}{2}\mathcal{B}(v,v)+\mathcal{B}(v,X)+\frac{1}{2}\mathcal{B}(X,X)],
    \end{equation}
    where $P(t)f$ satisfies the linear equation
    \begin{equation}
        \partial_tP(t)f- \Delta P(t)f=0,\quad P(0)f=f.
    \end{equation}
    By Lemma \ref{le;continuous dependence} with $L(v)=\mathcal{B}(v,X)$, it suffices to prove that there exists a positive time $T^*\leq 1$ such that
    \begin{gather}
        a\triangleq P(T^*)u_0-\frac{1}{2}\mathcal{B}(X,X)\in B(0,r)\subset C_{T^*}(H^s),\label{a}\\
        \sup_v\frac{\|\mathcal{B}(v,X)\|_{C_{T^*}(H^s)}}{\|v\|_{C_{T^*}(H^s)}}\leq\frac{1}{4},\label{B(v,X)}\\
        \|\mathcal{B}\|_{C_{T^*}(H^s)}\leq\frac{1}{4r}.\label{B(v,v)}
    \end{gather}
    For \eqref{a}, fix $T\in(0,1)$ small enough such that 
    \begin{align}\label{T2}
        C_2\left(T^{1-\frac{\rho_1}{2}}(1+T^{\frac{\rho_1}{2}})+T^{\frac{1-\rho_2}{2}}\right)\|X\|_{C_{1}(W^{\kappa-1})}^2\leq\frac{r}{2}.
    \end{align}
    Then, by Lemma \ref{estimate_B}, we have
    \begin{align*}
        \|a\|_{C_{T}(H^s)}&\leq \|P(\cdot)u_0\|_{C_{T}(H^s)}+\frac{1}{2}\|\mathcal{B}(X,X)\|_{C_{T}(H^s)}\\
        &\leq\|u_0\|_{H^s}+C_2\left(T^{1-\frac{\rho_1}{2}}(1+T^{\frac{\rho_1}{2}})+ T^{\frac{1-\rho_2}{2}}\right)\|X\|_{C_{1}(W^{\kappa-1})}^2\\
        &\leq r.
    \end{align*}
    For \eqref{B(v,X)}, fix $T\in(0,1)$ small enough such that
    \begin{align}\label{T3}
        C_3 \left(T^{\frac{1}{2}}(1+T^{\frac{1}{2}})+T^{1-\frac{\rho_1}{2}}(1+T^{\frac{\rho_1}{2}})+2 T^{\frac{1}{4}}\right)\|X\|_{C_{1}(W^{\kappa-1})}\leq\frac{1}{4}.
    \end{align}
    Then, from Lemma \ref{estimate_B} we obtain
    \begin{align*}
        &\frac{\|\mathcal{B}(v,X)\|_{C_{T}(H^s)}}{\|v\|_{C_{T}(H^s)}}\\
        \leq &C_3 \left(T^{\frac{1}{2}}(1+T^{\frac{1}{2}})+T^{1-\frac{\rho_1}{2}}(1+T^{\frac{\rho_1}{2}})+2 T^{\frac{1}{4}}\right)\|X\|_{C_{1}(W^{\kappa-1})}\\
        \leq &\frac{1}{4}.
    \end{align*}
    For \eqref{B(v,v)}, fix $T\in(0,1)$ small enough such that
    \begin{align}\label{T1}
       2C_1T^{\frac{1}{4}}( T^{\frac{1}{4}}+T^{\frac{3}{4}}+1)\leq\frac{1}{4r}.
    \end{align}
    Lemma \ref{estimate_B} then yields
    $$\|\mathcal{B}(v,v)\|_{C_T(H^s)}\leq 2C_1T^{\frac{1}{4}}( T^{\frac{1}{4}}+T^{\frac{3}{4}}+1)\|v\|_{C_T(H^s)}^2\leq\frac{1}{4r}\|v\|_{C_T(H^s)}^2.$$
    Taking $T^*\leq T$, we complete the proof.
\end{proof}

\subsection{The case of \texorpdfstring{$\eta\in(1/8,1/4]$}{eta in (1/8, 1/4]}}
We next turn to a rougher regime, where the classical interpretation of the nonlinear terms begins to break down. Indeed, since
$$
    X\in C_TW^{2\eta+\frac12-},
$$
the spatial derivative \(X_x\) has negative regularity, and hence the product \(X_x^2\) is no longer well-defined in the classical sense. This constitutes the first singular interaction that requires an enhancement of the stochastic convolution. We therefore define
$$
    X_x^2:=2X_x\prec X_x+X_x\circ X_x,
$$
where the resonant product \(X_x\circ X_x\) is constructed in Lemma~\ref{le;resonant}. With this enhanced second-order object at hand, all terms appearing in \eqref{v} can be given a rigorous meaning. The local well-posedness argument can then be carried out along the same lines as in Proposition~\ref{local_thm}, leading to the following result.

\begin{proposition}\label{local_thm;rough}
    Let $\eta\in(\frac{1}{8},\frac{1}{4}],\ \kappa\in(\frac{3}{2},2\eta+\frac{3}{2}),\ s\in[3-\kappa,\kappa)$ and $u_0\in H^s(\m T)$. Then there exist a positive $T^*=C(\Vert u_0\Vert_{H^s(\m T)},\Vert X\Vert_{C_1W^{\kappa-1}(\m T)},\Vert \p_x (X_x)^2 \Vert_{C_1W^{2\kappa-5}(\m T)})$ and a solution $u(t,x)$ of equation \eqref{eq;CH;diffusion} satisfying
    $$u(t,x)=X(t,x)+v(t,x),\quad t\in [0,T^*],$$
    where $v(t,x)$ satisfies \eqref{v} with $v(t,x)\in C_{T^*}H^s(\m T)~ a.s.$ and $X(t,x)$ satisfies \eqref{eq;X(t) decomposition} with $X(t,x)\in C_{T^*}W^{\kappa-1}(\m T)~ a.s.$. Moreover, we have $u\in C_{T^*}H^{\kappa-1}(\m T)$ a.s.. 
    \end{proposition}
    \begin{proof}
        Following the same lines as the proof of Proposition \ref{local_thm}, it suffices to show that there exists a positive $T^*\leq 1$ such that \eqref{a}-\eqref{B(v,v)} hold for $r\triangleq2\|u_0\|_{H^s}$. We deduce from Lemmas \ref{lemma;B1}, \ref{estimate_B} and \ref{le;resonant} that 
        \begin{align*}
        \|a\|_{C_{T}(H^s)}\leq& \|P(\cdot)u_0\|_{C_{T}(H^s)}+\frac{1}{2}\|\mathcal{B}(X,X)\|_{C_{T}(H^s)}\\
        \leq&\|u_0\|_{H^s}+C_2\left(T^{1-\frac{\rho_1}{2}}(1+T^{\frac{\rho_1}{2}})+ T^{\frac{1}{2}}\right)\|X\|_{C_{1}(W^{\kappa-1})}^2\\
        &+CT^{1-\frac{\rho}{2}}(1+T^{\frac{\rho}{2}})\|\partial_x(X_x^2)\|_{C_{1}(W^{2\kappa-5})},
    \end{align*}
    where $\rho=s-(2\kappa-3)\in(0,2)$. Fix $T\in(0,1)$ small enough such that
    $$C_2\left(T^{1-\frac{\rho_1}{2}}(1+T^{\frac{\rho_1}{2}})+ T^{\frac{1}{2}}\right)\|X\|_{C_{1}(W^{\kappa-1})}^2+CT^{1-\frac{\rho}{2}}(1+T^{\frac{\rho}{2}})\|\partial_x(X_x^2)\|_{C_{1}(W^{2\kappa-5})}\leq\frac{r}{2}$$
    and \eqref{T3}-\eqref{T1} hold. The proof is then completed by taking $T^*\leq T$.
    \end{proof}
 
\subsection{The case of \texorpdfstring{$\eta\in(0,1/8]$}{eta in (0, 1/8]}}
In the present regime, the enhanced structure introduced in the previous subsection is no longer sufficient to capture all the singular interactions. We therefore enlarge the stochastic data by incorporating the higher-order process \(X^{\dagger 2}\). More precisely, define
$$
    \mathbb X:=\{Y,X^{\dagger 2}\},
$$
where \(Y\) and \(X^{\dagger 2}\) are the stochastic processes constructed in Lemmas~\ref{X_bdd} and~\ref{le;high order}, respectively. For \(\kappa\in\mathbb R\), we equip \(\mathbb X\) with the norm
\begin{equation}
 \Vert \m X\Vert_{X^{\kappa}(t)}=\Vert Y\Vert_{\widetilde W^{\kappa-1}(t)}+\Vert X^{\dag 2}\Vert_{\widetilde W^{2\kappa-1}(t)} .
\end{equation}

Given the initial condition \(u_0\) and the enhanced data \(\mathbb X\), we say that a pair \((u,w)\) is a solution driven by \((u_0,\mathbb X)\) if
$$
    w(0)=u_0+Y(0),
$$
and
\begin{equation}\label{w}
\begin{aligned}
u=&Y+X^{\dagger 2}+w,\\
\partial_t w- \Delta w
=&P(D)\big(
(w+X^{\dagger 2})^2
+2(w+X^{\dagger 2})Y
+Y^2
+\frac12(w+X^{\dagger 2})_x^2
+(w+X^{\dagger 2})_xY_x
\big)\\
&-(w+X^{\dagger 2})(w+X^{\dagger 2})_x
-\partial_x\big((w+X^{\dagger 2})Y\big)
-YY_x .
\end{aligned}
\end{equation}
With this formulation, the singular stochastic contributions are encoded in the enhanced data \(\mathbb X\), while the remainder \(w\) satisfies an equation that can be treated by deterministic estimates. We are now in a position to state the corresponding local well-posedness result.

\begin{proposition}\label{local_thm;high}
    Let $\eta\in(0,\frac{1}{8}],\ \kappa\in(\frac{3}{2},2\eta+\frac{3}{2}),\ s\in[3-\kappa,\kappa)$, $u_0\in H^s(\m T)$, then there exist a positive $0<T^*\leq T=\sup\{t\geq0:\Vert u\Vert_{C_{t}W^{\kappa-1}}\}$ and unique solution $(u,v)\in C_{T^*}W^{\kappa-1}\times C_{T^*}H^s$  of \eqref{w} driven by $(u_0,\m X)$.
\end{proposition}
\begin{proof}
    By Lemmas \ref{X_bdd} and \ref{le;high order}, we have $\m X\in X^{\kappa}(t)$. Hence, it suffices to prove that there exists a positive $T^*\leq 1$ such that $w\in C_{T^*}(H^s)$ where $w$ satisfies \eqref{w}. 
    
    As in the proof of Proposition \ref{local_thm}, let $r\triangleq2\|u_0\|_{H^s}$ and 
    $$
    \mathcal{B}(f,g)\triangleq B(f,g)+B(g,f)+2\widetilde{B}(f,g)+\widetilde{B}(f_x,g_x).
    $$
    Then, $w$ can be expressed as
    \begin{align}
        w(t,x)=&P(t)u_0-[\frac{1}{2}\mathcal{B}(w,w)+\mathcal{B}(w,Y)+\mathcal{B}(w,X^{\dag 2})]\notag\\
        &-[\mathcal{B}(Y,X^{\dag 2})+\frac{1}{2}\mathcal{B}(X^{\dag 2},X^{\dag 2})+B(Y,Y)+\widetilde{B}(Y,Y)],
    \end{align}
    where $P(t)f$ satisfies the linear equation
    \begin{equation}
        \partial_tP(t)f- \Delta P(t)f=0,\quad P(0)f=f.
    \end{equation}
    By Lemma \ref{le;continuous dependence} with $L(w)=\mathcal{B}(w,Y)+\mathcal{B}(w,X^{\dag 2})$, it suffices to show that there exists a positive $T^*\leq 1$ such that
    \begin{gather}
        b\triangleq P(T^*)u_0-[\mathcal{B}(Y,X^{\dag 2})+\frac{1}{2}\mathcal{B}(X^{\dag 2},X^{\dag 2})+B(Y,Y)+\widetilde{B}(Y,Y)]\in B(0,r)\subset C_{T^*}(H^s),\label{b}\\
        \sup_w\frac{\|\mathcal{B}(w,Y)+\mathcal{B}(w,X^{\dag 2})\|_{C_{T^*}(H^s)}}{\|w\|_{C_{T^*}(H^s)}}\leq\frac{1}{4},\label{B(w,X)}\\
        \|\mathcal{B}\|_{C_{T^*}(H^s)}\leq\frac{1}{4r}.\label{B(w,w)}
    \end{gather}
    For \eqref{b}, fix $T\in(0,1)$ small enough such that 
    \begin{align*}
        &3C_2\left(T^{1-\frac{\rho_1}{2}}(1+T^{\frac{\rho_1}{2}})+ T^{\frac{1}{2}}\right)\Vert \m X\Vert_{X^{\kappa}(1)}^2
        \leq \frac{r}{2}.
    \end{align*}
    Then, Lemma \ref{estimate_B} gives
    \begin{align*}
        \|b\|_{C_{T}(H^s)}\leq& \|P(\cdot)u_0\|_{C_{T}(H^s)}+\|\mathcal{B}(Y,X^{\dag 2})+\frac{1}{2}\mathcal{B}(X^{\dag 2},X^{\dag 2})+B(Y,Y)+\widetilde{B}(Y,Y)\|_{C_{T}(H^s)}\\
        \leq&\|u_0\|_{H^s}+3C_2\left(T^{1-\frac{\rho_1}{2}}(1+T^{\frac{\rho_1}{2}})+ T^{\frac{1}{2}}\right)\Vert \m X\Vert_{X^{\kappa}(1)}^2\\
        \leq& r.
    \end{align*}
    For \eqref{B(w,X)}, fix $T\in(0,1)$ small enough such that
    \begin{align*}
        C_3 \left(T^{\frac{1}{2}}(1+T^{\frac{1}{2}})+T^{1-\frac{\rho_1}{2}}(1+T^{\frac{\rho_1}{2}})+2 T^{\frac{1}{4}}\right)\Vert \m X\Vert_{X^{\kappa}(1)}\leq\frac{1}{4}.
    \end{align*}
    Then, by Lemma \ref{estimate_B}, we have
    \begin{align*}
        &\frac{\|\mathcal{B}(w,Y)+\mathcal{B}(w,X^{\dag 2})\|_{C_{T^*}(H^s)}}{\|w\|_{C_{T^*}(H^s)}}\\
        \leq &C_3 \left(T^{\frac{1}{2}}(1+T^{\frac{1}{2}})+T^{1-\frac{\rho_1}{2}}(1+T^{\frac{\rho_1}{2}})+2 T^{\frac{1}{4}}\right)\Vert \m X\Vert_{X^{\kappa}(1)}\\
        \leq &\frac{1}{4}.
    \end{align*}
    Similar to the proof of Proposition \ref{local_thm}, fixing $T\in(0,1)$ small enough such that \eqref{T1} holds is sufficient to guarantee \eqref{B(w,w)}. Then there exists $(u,v)$ be the solution of \eqref{w} on $[0,T^*]$ driven by $(u_0,\m X)$ with $v\in C_TH^s$. Noting that $H^s\hookrightarrow W^{s-1/2}$, we have $(u,v)\in C_TW^{\kappa-1}\times C_TH^s$.
    \end{proof}

\subsection{The case of \texorpdfstring{$\eta\in(-1/12,0]$}{eta in (-1/12, 0]}}

    \par In this subsection, we consider the case $\eta\leq0$. In this case, the main singularity comes from $P(D)(w_x\circ Y_x)$, since $w\in C_TH^{2\eta+\frac{3}{2}-}$ and $Y\in C_TW^{2\eta+\frac{1}{2}-}$. This problem can be solved using the para-controlled ansatz. Assume that $w=u'\pprec\ml Q+u^\sharp$,  where $u'\in \widetilde{W}^{2\eta+\frac{1}{2}-}(T)$, $\ml Q\in \widetilde{W}^{2\eta+\frac{3}{2}-}(T)$ and $u^\sharp\in \widetilde{H}^{4\eta+2-}(T)$ will be chosen later so that $w$ solves \eqref{w}. Decomposing $(u'\pprec\ml Q)_x\circ Y_x$ into four parts and using the commutator estimates from Appendix B combined with Lemma \ref{le;high order}, we can prove that $(u'\pprec\ml Q)_x\circ Y_x\in C_TH^{4\eta-}$(see \eqref{estimate;commutator} for details). Since $u^\sharp\in \widetilde{H}^{4\eta+2-}(T)$ and $Y\in C_TW^{2\eta+\frac{1}{2}-}$, we have $u^\sharp_x\circ Y_x\in C_TH^{6\eta+\frac{1}{2}-}$ provided $\eta\in(-\frac{1}{12},0]$. This is precisely the reason why we require $\eta>-\frac{1}{12}$ in Theorem \ref{th;main results}.
  To determine the appropriate choices of \(u'\) and \(\mathcal Q\), we require the paracontrolled component \(u'\pprec\mathcal Q\) to capture the leading singular part of \(w\), while \(u^\sharp\) represents the more regular remainder. Combining \eqref{w} with the commutator estimates from Appendix~B, we obtain
    \begin{align*}
        \ml Lu^\sharp=&\ml Lw-\ml L(u'\pprec\ml Q)\\
        =&P(D)(w_x\circ Y_x)-(w+Y+X^{\dag 2})\pprec Y_x-u'\pprec\ml{LQ}+C_TH^{4\eta-}.
    \end{align*}
 We therefore choose
\begin{equation}\label{eq:choice-uQ}
u'=-(w+Y+X^{\dagger2}),
\qquad
\mathcal L\mathcal Q=Y_x.
\end{equation}
With this choice, the two paraproduct terms cancel, and the remaining singular contribution \(P(D)(w_x\circ Y_x)\) can be controlled through the enhanced resonant product \(\mathcal Q_x\circ Y_x\). Consequently,
$$
    u^\sharp\in\widetilde H^{4\eta+2-}(T),
$$
which closes the paracontrolled ansatz.
\par We now introduce the functional framework for the fixed-point argument. Define the enhanced data by
\[
    \m X
    :=
    \big\{Y,X^{\dagger2},\mathcal Q_x\circ Y_x\big\},
\]
and define its norm by
\begin{equation}
    \Vert \m X\Vert_{X_{[s,t]}^\kappa}
    =
    \Vert Y\Vert_{\widetilde W^{\kappa-1}([s,t])}
    +\Vert X^{\dag 2}\Vert_{\widetilde W^{2\kappa-1}([s,t])}
    +\Vert \p_x \ml Q\circ \p_x Y\Vert_{\widetilde W^{2\kappa-3}([s,t])}.
\end{equation}
In particular, we write
\[
    \Vert \m X\Vert_{X_T^\kappa}
    =
    \Vert \m X\Vert_{X_{[0,T]}^\kappa}.
\]
Given the initial condition \(u_0\) and the enhanced data \(\mathbb X\), we say that a triple \((u,u',u^\sharp)\) is a paracontrolled solution driven by \((u_0,\mathbb X)\) if
\begin{equation}\label{eq:paracontrolled-decomposition}
u=-u'
=Y+X^{\dagger2}+w
=Y+X^{\dagger2}+u'\pprec\mathcal Q+u^\sharp,
\end{equation}
and \(u^\sharp\) satisfies
    \begin{equation}\label{eq;paracontrolled equation}
    	\begin{aligned}
    		u'=&-Y-X^{\dag2}-u'\pprec\ml Q-u^\sharp\\
    		\ml L u^\sharp=& -[\frac{1}{2}\p_x(w+Y+X^{\dag 2})^2-(w+Y+X^{\dag 2})\prec \p_xY]\\
            &+P(D)\big((w+Y+X^{\dag 2})^2
            +\frac12 (w_x+X_x^{\dag 2})^2+2(w_x+X_x^{\dag 2})Y_x)+u'\prec \p_xY-\ml L(u'\pprec \ml Q)\big)
    	\end{aligned}
    \end{equation}
together with the prescribed initial condition.
\begin{proposition}\label{prop;paracontrolled well-posedness}
    	Let $\eta\in(-1/12,0]$, $\kappa\in(\frac{5}{4}-\eta, \frac{3}{2}+2\eta)$, $\m X=\{Y,X^{\dag 2},\ml Q_x\circ X_x\}\in \widetilde W^{\kappa-1}(1)\times \widetilde W^{2\kappa-1}(1)\times \widetilde W^{2\kappa-3}(1)$, then there exists $T^*$ and unique triple $(u,u',u^\sharp)\in \widetilde W^{\kappa-1}(T^*)\times \widetilde W^{\kappa-1}(T^*)\times \widetilde H^{2\kappa-1}(T^*) $ be the solution of \eqref{eq;paracontrolled equation} driven by $(u_0,\m X)$.
    \end{proposition}
\begin{proof}
    By Lemmas \ref{X_bdd} and \ref{le;high order}, we have $Y\in \widetilde{W}^{\kappa-1}(1)$ and $X^{\dag 2}\in\widetilde{W}^{2\kappa-1}(1)$. Hence, it suffices to prove that there exist a positive $T^*\leq 1$ and $(u^\sharp,u')\in \widetilde{H}^{2\kappa-1}(T^*)\times\widetilde{W}^{\kappa-1}(T^*)$ such that equation \eqref{w} has a solution $w$ satisfying $w=u'\pprec\ml Q+u^\sharp\triangleq f(u^\sharp,u')$. Note that if $(u^\sharp,u')$ satisfies
    \begin{equation}\label{eq;u',u^sharp}
        \left\{
        \begin{aligned}
            &u^\sharp=a+h(u^\sharp)+g(u')+\ml G(f(u^\sharp,u'))+\ml B(f(u^\sharp,u'),f(u^\sharp,u')),\\
            &u'=-(Y+X^{\dag 2})-f(u^\sharp,u')
        \end{aligned}
        \right.
    \end{equation}
    with $u^\sharp(0)=u_0+Y(0),\ u'=-u_0$, where
    \begin{equation*}
        \left\{
        \begin{aligned}
            &\ml La =[(Y+X^{\dag 2})\prec  Y_x-\frac{1}{2}\p_x(Y+X^{\dag 2})^2]+P(D)[(Y+X^{\dag 2})^2+\frac{1}{2}(X^{\dag 2}_x)^2+Y_x X^{\dag 2}_x],\\
            &\ml Lh(u^\sharp)=P(D)(u^\sharp_x\circ Y_x),\\
            &\ml Lg(u')=[u'\pprec Y_x-\ml L(u'\pprec\ml Q)]+(u'\prec Y_x-u'\pprec Y_x)+P(D)[(u'\pprec\ml Q)_x\circ Y_x],\\
            &\ml L\ml G(f(u^\sharp,u'))=-[w(Y+X^{\dag 2})]_x+w\prec Y_x\\
            &~~~~~~~~~~~~~~~~~~~~~+P(D)[2w(Y+X^{\dag 2})+w_x X^{\dag 2}_x+w_x Y_x-w_x\circ Y_x],\\
            &\ml L\ml B(f(u^\sharp,u'),f(u^\sharp,u'))=-ww_x+P(D)(w^2+\frac{1}{2}w_x^2)
        \end{aligned}
        \right.
    \end{equation*}
    with $a(0)=u_0$ and $h(u^\sharp)(0)=g(u')(0)=\ml G(f(u^\sharp,u'))(0)=\ml B(f(u^\sharp,u'),f(u^\sharp,u'))(0)=0$, 
    then $w=u'\pprec\ml Q+u^\sharp$ solves \eqref{w}. In what follows, we prove that there exist a positive $T^*\leq 1$ and a unique $(u^\sharp,u')\in \widetilde{H}^{2\kappa-1}(T^*)\times\widetilde{W}^{\kappa-1}(T^*)$ satisfying \eqref{eq;u',u^sharp}.
    
    Defining 
    \begin{align*}
		\ml E^\kappa_T:=\{(u,v)\in \widetilde{H}^{2\kappa-1}(T)\times \widetilde{W}^{\kappa-1}(T): 2\Vert u\Vert_{\widetilde{H}^{2\kappa-1}(T)}+\Vert v\Vert_{\widetilde{W}^{\kappa-1}(T)}<\infty\},
    \end{align*}
    we aim to show that there is a positive $T^*\leq 1$ such that the conditions of Lemma \ref{le;Picard th;specific} hold when we take $\dot E=\widetilde{H}^{2\kappa-1}(T^*),\ X=\widetilde{H}^{\kappa}(T^*),\ E=\widetilde{W}^{\kappa-1}(T^*)$ and $\ml E=\ml E_{T^*}$ therein. Assume that $T\in(0,1]$. Taking $\rho=2-\frac{2\eta+\frac{1}{2}-(\kappa-1)}{2}\in(0,2)$, we deduce from Lemma \ref{le;est;paraproduct;modifed} that 
    \begin{align*}
	\Vert f(u^\sharp,u')\Vert_{\widetilde{H}^{\kappa}(T)}\leq& 
        \Vert u^\sharp\Vert_{\widetilde{H}^{\kappa}(T)}+C  T^{1-\frac{\rho}{2}}\Vert u'\Vert_{\widetilde{W}^{\kappa-1}(T)}\left(\Vert \ml Q\Vert_{\widetilde{H}^{3-\rho}(T)}+\Vert \ml{LQ}\Vert_{\widetilde{H}^{\kappa-\rho}(T)}\right)\\
	\leq& \Vert u^\sharp\Vert_{\widetilde{H}^{2\kappa-1}(T)}+C  T^{1-\frac{\rho}{2}}\Vert u'\Vert_{\widetilde{W}^{\kappa-1}(T)}\Vert Y \Vert_{\widetilde{H}^{\kappa+1-\rho}(T)}.
    \end{align*}
    Taking $T\in(0,1]$ small enough, we can ensure that $C T^{1-\frac{\rho}{2}}\Vert Y \Vert_{\widetilde{H}^{\kappa+1-\rho}(T)}<\frac{1}{30}$. 

    For $\eta\in(-\frac{1}{12},0],\ \kappa\in(\frac{5}{4}-\eta, \frac{3}{2}+2\eta),\ \varepsilon_1\in(\max(0,\frac{8}{3}-2\kappa),4\eta+3-2\kappa)$, it follows from Lemmas \ref{le;heat flow;time smoothing} and \ref{le;commutater;D} that 
    \begin{align*}
        &\Vert (a,-(Y+X^{\dag 2}))\Vert_{\ml E_{T}}\\
        \leq& C\Vert u_0\Vert_{H^{2\kappa-1}}+C T^{\frac{\varepsilon_1}{2}}\Vert \ml La\Vert_{ C_T H^{2\kappa-3+\varepsilon_1}}+\Vert Y+X^{\dag 2}\Vert_{\widetilde{W}^{\kappa-1}(T)}\\
        \leq&C T^{\frac{\varepsilon_1}{2}}\Vert \m X\Vert_{X^{\kappa+\frac{\varepsilon_1}{2}}(T)}^2
        +C\Vert u_0\Vert_{H^{2\kappa-1}}+\Vert \m X\Vert_{X^{\kappa}(T)}.
    \end{align*}
    Setting $r=16\max(1,C)\cdot\max(\Vert u_0\Vert_{H^{2\kappa-1}},\Vert \m X\Vert_{X^{\kappa}(1)})$, we can choose $T\in(0,1]$ small enough such that $\Vert (a,-(Y+X^{\dag 2}))\Vert_{\ml E_{T}}\leq\frac{r}{4}$.

    Note that $\kappa\in(-\eta+\frac{5}{4},2\eta+\frac{3}{2})\subset(-4\eta+1,2\eta+\frac{3}{2})$. For $\varepsilon_2\in(\max(0,\frac{4-3\kappa}{2}),2\eta+\frac{3}{2}-\kappa)$, we deduce from Lemmas \ref{le;heat flow;time smoothing}, \ref{le;high order}, \ref{le;commutator;resonant} and \ref{le;commutator;modfied paraproduct, L} that 
    \begin{align*}
        \Vert g(u')\Vert_{\widetilde{H}^{2\kappa-1}(T)}
        \leq& C T^{\frac{\varepsilon_2}{2}}\Vert \ml Lg(u')\Vert_{ C_T H^{2\kappa-3+\varepsilon_2}}\\
        \leq&C T^{\frac{\varepsilon_2}{2}}\left(\Vert u'\Vert_{\widetilde{W}^{\kappa-1}(T)}\left(\Vert\ml Q\Vert_{ C_T H^{\kappa+\varepsilon_2}}+\Vert Y_x\Vert_{ C_T H^{\kappa-2+\varepsilon_2}}\right)\right)\\
        &+C T^{\frac{\varepsilon_2}{2}}\Vert (u'\pprec\ml Q)_x\circ Y_x\Vert_{ C_T H^{2\kappa-4+\varepsilon_2}},
    \end{align*}
    where
    \begin{align}
        &\Vert(u'\pprec\ml Q)_x\circ Y_x\Vert_{ C_T H^{2\kappa-4+\varepsilon_2}}\notag\\
        \leq&\Vert[(u'\pprec\ml Q)_x-u'\pprec\ml Q_x]\circ Y_x\Vert_{ C_T H^{2\kappa-4+\varepsilon_2}}+\Vert [u'\pprec\ml Q_x-u'\prec\ml Q_x]\circ Y_x\Vert_{ C_T H^{2\kappa-4+\varepsilon_2}}\notag\\
        &+\Vert (u'\prec\ml Q_x)\circ Y_x-u'(\ml Q_x\circ Y_x)\Vert_{ C_T H^{2\kappa-4+\varepsilon_2}}+\Vert u'(\ml Q_x\circ Y_x)\Vert_{ C_T H^{2\kappa-4+\varepsilon_2}}\notag\\
        \leq&C\Vert u'\Vert_{\widetilde{W}^{\kappa-1}(T)}\left(\Vert\ml Q\Vert_{ C_T H^{\kappa+\varepsilon_2}}\Vert Y_x\Vert_{ C_T W^{\kappa-2+\varepsilon_2}}+\Vert \ml Q_x\circ Y_x\Vert_{ C_T W^{2\kappa-3+2\varepsilon_2}}\right).\label{estimate;commutator}
    \end{align}
    Thus, we obtain $\Vert g(u')\Vert_{\widetilde{H}^{2\kappa-1}(T)}\leq C T^{\frac{\varepsilon_2}{2}}\Vert \m X\Vert_{X^{\kappa+\varepsilon_2}(1)}\Vert u'\Vert_{\widetilde{W}^{\kappa-1}(T)}$. Taking $T\in(0,1]$ small enough, we can ensure that $\Vert g(u')\Vert_{\widetilde{H}^{2\kappa-1}(T)}\leq\frac{1}{16}\Vert u'\Vert_{\widetilde{W}^{\kappa-1}(T)}$.

    Similarly, for $\varepsilon_3\in(\max(0,4-3\kappa),2\eta+\frac{3}{2}-\kappa)$, we can choose $T\in(0,1]$ small enough such that
    \begin{align*}
        \Vert h(u^\sharp)\Vert_{\widetilde{H}^{2\kappa-1}(T)}
        \leq& C T^{\frac{\varepsilon_3}{2}}\Vert \ml Lh(u^\sharp)\Vert_{ C_T H^{2\kappa-3+\varepsilon_3}}
        \leq C T^{\frac{\varepsilon_3}{2}}\Vert u^\sharp_x\Vert_{\widetilde{H}^{2\kappa-2}(T)}\Vert Y_x\Vert_{ C_T W^{\kappa-2+\varepsilon_3}}\\
        \leq&C T^{\frac{\varepsilon_3}{2}}\Vert u^\sharp\Vert_{\widetilde{H}^{2\kappa-1}(T)}\Vert \m X\Vert_{X^{\kappa+\varepsilon_3}(1)}\leq\frac{1}{16}\Vert u^\sharp\Vert_{\widetilde{H}^{2\kappa-1}(T)},\\
        \Vert \ml G(f(u^\sharp,u'))\Vert_{\widetilde{H}^{2\kappa-1}(T)}
        \leq& C T^{\frac{\varepsilon_3}{2}}\Vert \ml L\ml G(f(u^\sharp,u'))\Vert_{ C_T H^{2\kappa-3+\varepsilon_3}}\\
        \leq& C T^{\frac{\varepsilon_3}{2}}\Vert f(u^\sharp,u')\Vert_{\widetilde{H}^{\kappa}(T)}\Vert \m X\Vert_{X^{\kappa+\varepsilon_3}(1)}\\
        \leq&\frac{1}{16}\Vert f(u^\sharp,u')\Vert_{\widetilde{H}^{\kappa}(T)}
    \end{align*}
    and
    \begin{align*}
        \Vert \ml B(f(u^\sharp,u'),f(u^\sharp,u'))\Vert_{\widetilde{H}^{2\kappa-1}(T)}
        \leq& C T^{\frac{\varepsilon_3}{2}}\Vert \ml L\ml B(f(u^\sharp,u'),f(u^\sharp,u'))\Vert_{ C_T H^{2\kappa-3+\varepsilon_3}}
        \\
        \leq&C T^{\frac{\varepsilon_3}{2}}\Vert f(u^\sharp,u')\Vert_{\widetilde{H}^{\kappa}(T)}^2\leq\frac{1}{16r}\Vert f(u^\sharp,u')\Vert_{\widetilde{H}^{\kappa}(T)}^2.
    \end{align*}
    
    Taking $T^*\leq T$, we deduce from Lemma \ref{le;Picard th;specific} that there exists a unique $(u^\sharp,u')\in \widetilde{H}^{2\kappa-1}(T^*)\times\widetilde{W}^{\kappa-1}(T^*)$ satisfying \eqref{eq;u',u^sharp}.
\end{proof}
\subsection{Continuous dependence}
In this subsection, we give the continuously dependent with $(\m X,u_0)$. For $\eta\in(0,\infty)$, since the pair $(u,v)$ satisfies the structure
\begin{equation}
    u=\sum_{\tau\in\ml T'}c(\tau)X^\tau+v,\quad v=a+L(v)+\ml B(v,v),
\end{equation}
it can be treated by the same argument as in the proof of Theorem 3.1 in \cite{zhang2026generalizedframeworksingularfractional}. We introduce the following proposition without proof.
\begin{proposition}
    Let $\eta\in(0,\infty)$ and $u_0\in H^s$. Suppose that $\kappa$, $s$, 
    and $u_0$ satisfy the assumptions of Theorem~\ref{th;main results}. 
    Let $\bar T$ be the maximal existence time, then $\bar T$ is independent of the choice of $s$ and is characterized by
    \begin{equation}
        \bar T
        =
        \sup\left\{
            t>0:\|v\|_{C_tH^s}<\infty
        \right\}.
    \end{equation}

    Moreover, let $u^\sigma$ be the solution to the regularized equation
    \begin{equation}\label{eq:regularized-solution}
        \begin{cases}
        \partial_t u^\sigma
        +u^\sigma\partial_xu^\sigma
        -\Delta u^\sigma
        =
        P(D)\left(
            (u^\sigma)^2+\dfrac12(\partial_xu^\sigma)^2
        \right)
        +(1-\partial_x^2)^{-\eta}\xi_\sigma,\\[1mm]
        u^\sigma(0)=\phi_\sigma*u_0.
        \end{cases}
    \end{equation}
    Then there exists a family of times $T_\sigma>0$ such that
    \[
        T_\sigma\longrightarrow\bar T
        \qquad\text{as }\sigma\to0,
    \]
    and
    \begin{equation}
        \lim_{\sigma\to0}
        \|u^\sigma-u\|_
        {C([0,T_\sigma\wedge\bar T);W^{\kappa-1})}
        =0.
    \end{equation}
\end{proposition}
 
 \par Now we try to show the paracontrolled solution $(u_n',u^\sharp_n)$ driven by smooth data $(\m X^n,u_0^n)$ is converge to $(u',u^\sharp)$ in space $\ml E_T$ until blow up time. Assume that $(u,u',u^\sharp)$ is the solution of \eqref{eq;paracontrolled equation} driven by $(u_0,\m X)$,
 \begin{proposition}\label{prop;continuous dependence}
The maximal time of existence $\bar T$ of such a solution in Proposition  \ref{prop;paracontrolled well-posedness} does not depend on $s$ and satisfies
		\begin{equation}
			\bar T=\sup\{t>0; \Vert (u^\sharp,u')\Vert_{\ml E_t^\kappa}<\infty\}.
		\end{equation}
    
  	Moreover, assume that  $\Vert u_0^n-u_0\Vert_{W^s}\to0$ and  $\Vert \m X^n-\m X\Vert _{X_T^{\kappa}}\to0$ for any $T>0$,  and $(u,u',u^\sharp)$ be the solution of \eqref{eq;paracontrolled equation}  driven by $(\m X,u_0)$ on $[0,\bar T)$, 
    then there exists $T_n$ satisfying $\lim_{n\to\infty} T_n=\bar T$ and for any $n\in \m N$, we have $(u^n,u^{n,}{'},u^{n,\sharp})$ is the solution of \eqref{eq;paracontrolled equation} driven by $(\m X^n,u^n_0)$ on $[0,T_n]$. Moreover, we have $\Vert u^n-u\Vert_{C([0,T_n\wedge \bar T);W^{\kappa-1})}\to0$ as $n\to\infty$. 
 \begin{proof}
     We firstly prove the solution exists until blow up time. Assume that we already have the solution $(u,u',u^\sharp)$ on $[0,h]$ for some $h>0$, and consider the time interval $[h,h+1]$. The initial condition has enough regularity by following computation
     \begin{equation}\label{eq;endpoint}
         \begin{aligned}
			\tilde u^\sharp(0)=&\tilde u(0)-\sum_{\tau\in\ml T'}c(\tau) \tilde X^\tau-\tilde u'\pprec \tilde {\ml Q}(0)\\
			=&u(h)-\sum_{\tau\in\ml T'}c(\tau) X^\tau(h)-u'(h)\prec \ml Q(h)
		\end{aligned}
     \end{equation}
		Note that $u'\pprec \ml Q(h)-u'\prec\ml Q(h)\in H^{2\kappa-1}$, we have $\tilde u^\sharp(0)\in H^{2\kappa-1}$. As a consequence, let $\tilde M=16\max(1,C)\cdot \max(\Vert \tilde u^\sharp(0)\Vert_{H^{2\kappa-1}},\Vert \m X\Vert_{X^\kappa_{[h,h+1]}})$, then we obtain unique para-controlled solution $\tilde u$ on $[0,h(\tilde M)]$. We then extend $u$ and $u'$ from $[0,h]$ to $[0,h+h(\tilde M)]$ by setting $u(t)=\tilde u(t-h)$ for $t\geq h$. Then by definition $u'\in \widetilde W^{\kappa-1}(h+h(\tilde M))$. For the remainder term, note that $u^{\ml Q}-u'\pprec_h\ml Q\in \widetilde H^{2\kappa-1}$ where
		\begin{align*}
			u'\pprec_h \ml Q(t)=\sum_{j}\int_{-\infty}^t 2^{2 j}\varphi(2^{2 j}(t-s))S_{j-1}u'(s\vee h)\mathrm{d}s\Delta_j \ml Q(t).
		\end{align*}
		It suffice to show that $(u'\pprec_h\ml Q-u'\pprec \ml Q)|_{[h,h+h(\tilde M)]}\in \widetilde H^{2\kappa-1}([h,h+h(\tilde M)])$.
        Noting that $(u'\pprec_h\ml Q-u'\pprec\ml Q)(h)\in H^{2\kappa-1}$ and for $t\in [h,h+h(\tilde M)]$ we get
        \begin{align*}
            \ml L(u'\pprec _h \ml Q-u'\pprec \ml Q)(t)=(u'\pprec_h \p_x X^{\tau^*}-u'\pprec \p_x X^{\tau^*})(t)+H^{2\kappa-3}
        \end{align*}
        By Lemma B.3 in \cite{zhang2026wellposednessregularityfractionalrough}, we have that 
$u' \pprec \partial_x X^{\tau} - u' \prec \partial_x X^{\tau} 
\in C_{h+h(\tilde M)} H^{2\kappa-3}$. 
By the same arguments it follows that also 
$(u' \prec_h \partial_x X^{\tau} - u' \prec \partial_x X^{\tau})|_{[h,h+h(\tilde M)]} 
\in C_{[h,h+h(\tilde M)]} H^{2\kappa-3}$; 
by Lemma \ref{le;heat flow;time smoothing}, we obtain 
$\tilde u^\sharp \in {\widetilde{ H}}^{2\kappa-1}([h,h+h(\tilde M)])$.
\par Now we prove that the continuous dependence with $(\m X,u_0)$ until blow up time. Let $T<\bar T$, we have $\Vert (u^\sharp,u')\Vert_{\ml E^\kappa_T}<\infty$.  Assume that $u_0^{n}\to u_0$ in $W^s$, and $\m X^{n}\to \m X$ in $X_T^\kappa$. More precisely, let $\epsilon_n,\delta_n\to0$ be such that \begin{align*}     \Vert u_0^{n}-u_0\Vert_{W^s}\leq\epsilon_n,\quad     \Vert \m X^{n}-\m X\Vert_{X^\kappa_T}\leq\delta_n. \end{align*}
 Let $M=\max(1,C)\cdot\max\{\Vert (u^\sharp,u')\Vert_{\ml E^\kappa_T},\Vert \m X\Vert_{X_T^\kappa}\}$. And $r=16(M+1)$, then  Theorem \ref{local_thm},  for any $n\in \m N$, there exits a positive $T=T(M)$ and a unique triple $(u^{n},u^{n,}{'},u^{n,\sharp})\in \widetilde W^{\kappa-1}(T)\times\widetilde W^{\kappa-1}(T)\times \widetilde H^{2\kappa-1}(T)$ satisfying the paracontrolled equation \eqref{eq;paracontrolled equation} driven by $(u_0^{n},\m X^{n})$. Consider $u^{n,}{'}(T(M))$ and $u^{n,\sharp}(T(M))$, noting that
     \begin{equation}\label{eq;u',u^sharp;app}
        \left\{
        \begin{aligned}
            &u^{n,\sharp}
            =a_n+h_n(u^{n,\sharp})+g_n(u^{n,'})
            +\ml G_n(f_n(u^{n,\sharp},u^{n,'}))
            +\ml B(f_n(u^{n,\sharp},u^{n,'}),
            f_n(u^{n,\sharp},u^{n,'})),\\
            &u^{n,'}
            =-(X_{n}+X_{n}^{\dag 2})
            -f_n(u^{n,\sharp},u^{n,'}).
        \end{aligned}
        \right.
\end{equation}
 by Lemma \ref{le;continuous dependence}, there exists $N\in\m N$ such  that for any $n\geq N$, we have
\begin{equation}\label{eq;k=0;2}
     \Vert (u^\sharp,u')- (u^{n,\sharp},u^{n,}{'})\Vert_{\ml E_{T(M)}^\kappa}
 \leq C(1+M+M^2)\delta_n \big( 1+\Vert \m X\Vert_{X_T^\kappa} +\Vert \m X\Vert_{X_T^\kappa}^2 \big).+\delta_n+\epsilon_n
\end{equation}
Calculating by $\Vert (u^{n,\sharp},u^{n,}{'})\Vert_{\ml E_{T(M)}^\kappa}\leq \Vert (u^\sharp,u')\Vert_{\ml E_{T(M)}^\kappa}+\Vert (u^{n,\sharp},u^{n,}{'})-(u^\sharp,u')\Vert_{\ml E_{T(M)}^\kappa}$, we have
\begin{equation}\label{eq;k=0}
    \Vert (u^{n,\sharp},u^{n,}{'})\Vert_{\ml E_{T(M)}^\kappa}\leq M+(1+C(M))\delta_n+\epsilon_n
\end{equation}
Let $K=[T/T(M)]+1$ and define
\begin{equation}\label{eq;condition,iteration}
    \epsilon_{n,0}=\epsilon_n,\qquad
    \epsilon_{n,k+1}
    =\epsilon_{n,k}+(1+C(M))\delta_n,
    \quad k=0,\ldots,K-1.
\end{equation} Define $(u^{n,k},u^{n,k,}{'},u^{n,k,\sharp})$ be the solution of \eqref{eq;paracontrolled equation} driven by $(\m X^{n,k},u^{n,k-1,\sharp}(T(M)))$ on $[0,T(M)]$, where $u^{n,0,\sharp}=u_0$  and and
\begin{align*}
     \m X^{n,k}
     =\{&Y_{n}(\cdot+kT(M)),
     X^{\dag 2}_{n}(\cdot+kT(M)),(\p_x\ml Q_{n}\circ\p_xY_{n})
     (\cdot+kT(M))\}.
\end{align*}
Particularly, $(u^{n},u^{n,}{'},u^{n,\sharp})=(u^{n,0},u^{n,0,}{'},u^{n,0,\sharp})$ on $[0,T(M)]$. We use the induced argument to prove that for any $k$, we have
\begin{equation}\label{eq;induced claim}
\begin{aligned}
     &\Vert (u^{n,k,\sharp},u^{n,k,'})
     \Vert_{\ml E_{T(M)}^\kappa}\leq M+1,\\
     &\Vert (u^{n,k,\sharp},u^{n,k,'})
     -(u^\sharp(\cdot+kT(M)),u'(\cdot+kT(M)))
     \Vert_{\ml E_{T(M)}^\kappa}
     \leq\epsilon_{n,k}.
\end{aligned}
\end{equation}
 Indeed, the case $k=0$ follows from \eqref{eq;k=0} and \eqref{eq;k=0;2}. Assume \eqref{eq;induced claim}  is hold for $k=m$, then $\Vert u^{n,m+1,\sharp}(0)\Vert_{H^{2\kappa-1}}\leq M+1$, which ensure that $(u^{n,m+1},u^{n,m+1,}{'},u^{n,m+1,\sharp})$ is exists on $[0,T(M)]$. To estimate $\Vert u^{n,m+1,\sharp}(\cdot)-u^{\sharp}(\cdot+m(T(M)))\Vert_{\widetilde H^{2\kappa-1}(T(M))}$, we have
\begin{align*}
     \Vert u^{n,m+1,\sharp}(0)
     -u^{\sharp}((m+1)T(M))\Vert_{H^{2\kappa-1}}
     &\leq 
     \Vert u^{n,m,\sharp}(T(M))
     -u^{\sharp}((m+1)T(M))\Vert_{H^{2\kappa-1}}\\
     &\leq \epsilon_{n,m}.
\end{align*}
 Similar as the case of $k=0$, we have
\begin{align*}
      &\Vert (u^{n,m+1,\sharp},u^{n,m+1,'})
      -(u^\sharp(\cdot+(m+1)T(M)),u'(\cdot+(m+1)T(M)))
      \Vert_{\ml E_{T(M)}^\kappa}\\
      &\qquad\leq
      (1+C(M))\delta_n+\epsilon_{n,m}
      =\epsilon_{n,m+1}.
\end{align*}
then $ \Vert (u^{n,m+1,\sharp}(\cdot),u^{n,m+1,}{'}(\cdot))\Vert_{\ml E_{T(M)}^\kappa} \leq M+1$, and the induced claim is hold. Noting that, for
$t\in[kT(M),T\wedge(k+1)T(M)]$,
\[
    (u^{n},u^{n,'},u^{n,\sharp})(t)
    =
    (u^{n,k},u^{n,k,'},u^{n,k,\sharp})(t-kT(M)),
\]
we obtain a solution of \eqref{eq;paracontrolled equation} driven by
$(\m X^n,u_0^n)$ on $[0,T]$.  Since $K$ is fixed and $\epsilon_n,\delta_n\to0$, we have
\[
    \max_{0\leq k\leq K}\epsilon_{n,k}\to0.
\]
Therefore, for every $T<\bar T$,
\begin{equation}
     \Vert (u^\sharp,u')-
     (u^{n,\sharp},u^{n,'})
     \Vert_{\ml E_T^\kappa}
     \to0.
\end{equation}
By the paracontrolled decomposition and the convergence
$\m X^n\to\m X$ in $X_T^\kappa$, we further obtain
\[
    \Vert u^n-u\Vert_{C_TW^{\kappa-1}}\to0.
\]
Since $T<\bar T$ is arbitrary, we may choose the existence times
$T_n$ such that $T_n\to\bar T$ and
\[
    \Vert u^n-u
    \Vert_{C([0,T_n\wedge\bar T);W^{\kappa-1})}\to0.
\]
This completes the proof.
 \end{proof}
 \end{proposition}

 \begin{proof}[proof of Theorem \ref{th;main results}]
 	We prove Theorem \ref{th;main results} by combining the continuity of the solution map with the convergence of the enhanced stochastic process. It's the directly conclusion of Proposition \ref{prop;continuous dependence}.  Therefore, it remain to  consider the regime of $\beta\in(-1/12,0]$. 
    \par Let $\sigma_n$ be any sequence converge to 0, and define $\m X^{\sigma_n}=\{Y_{\sigma_n},X^{\dag 2}_{\sigma_n},\p_x\ml Q_{\sigma_n}\circ\p_xX_{\sigma_n}\}$, and $u_0^{\sigma_n}=\phi_{\sigma_n}\ast u_0$, by Proposition \ref{prop;continuous dependence}, $(u^{\sigma_n,}{'},u^{\sigma_n,\sharp})\to (u',u^\sharp)$ in $\ml E_T$ in probability sense as $n\to\infty$. Noting that $u^{\sigma_n}=Y_{\sigma_n}+X^{\dag2}_{\sigma_n}+u^{\sigma_n,}{'}\pprec\ml Q_{\sigma_n}+u^{\sigma_n,\sharp}$ is actually the solution of \eqref{eq;CH;diffusion;app} for $\sigma=\sigma_n$. Then we prove the Theorem \ref{th;main results} is hold for $\eta>-1/12$ and finish our proof.   
 \end{proof}
 \section{Global well-posedness}\label{sec;global well-posedness}
Throughout this section, we work with zero-mean functions on the torus, i.e. $\int_{\m T}u(x)\,\mathrm{d}x=0$,  For such functions, the homogeneous Sobolev norm $\|\cdot\|_{\dot H^s}$ and the inhomogeneous Sobolev norm $\|\cdot\|_{H^s}$
  are equivalent when $s\geq0$, a fact that will be used repeatedly in the sequel. In order to establish the global existence of solutions to the Cauchy problem \eqref{eq;CH;diffusion}, we begin by obtaining a priori energy estimates. 
 \subsection{Energy estimates}
We first present a lemma that will be useful in what follows. \begin{lemma}\label{le;v_estimates}\cite{bahouriFourierAnalysisNonlinear2011}
		Let $v$ be the solution in $\ml C([0,T];\ml S'(\mathbb{T}^d))$ of the Cauchy problem
		\begin{equation}\label{eq;heat equation;general}
			\left\{
			\begin{aligned}
				&\p_t v- \Delta v=f,\\
				&v|_{t=0}=v_0,\\		
			\end{aligned}
			\right.
		\end{equation}
		with $f\in L^2([0,T]; \dot H^{s-1}(\m T^d))$ and $v_0\in \dot H^s(\mathbb{T}^d)$. Then, 
		\begin{align*}
			v\in (\cap_{p=2}^\infty L^p([0,T]; \dot H^{s+\frac{2}{p}}(\mathbb{T}^d)))\cap \ml C([0,T]; \dot H^s(\mathbb{T}^d)).
		\end{align*}
		Moreover, we have the following estimates:
		\begin{align*}
			\Vert v(t)\Vert_{\dot H^s}^2+2\int_0^t\Vert \nabla v(t')\Vert_{\dot H^s}^2\mathrm{d}t'=\Vert v_0\Vert_{\dot H^s}^2+2\int_0^t\langle f(t'),v(t')\rangle_s\mathrm{d}t',\\
		\left(\sum_{k\neq 0}|
            k|^{2s}(\sup_{0\leq t'\leq t}\vert \hat v(t',k)\vert)^2\right)^\frac{1}{2} \leq \Vert v_0\Vert_{\dot H^s}+\frac{1}{\sqrt{2}}\Vert f\Vert_{L_T^2(\dot H^{s-1})},\\
            \Vert v(t)\Vert_{L_T^p(\dot   H^{s+\frac{2}{p}})}\leq \left( \Vert v_0\Vert_{\dot H^s}+\Vert f\Vert_{L_T^2(\dot H^{s-1})}\right)
		\end{align*}
		with $\langle a,b\rangle_s=\int\langle \xi\rangle^{2s}\hat a(\xi)\overline{\hat b(\xi)}\mathrm{d}\xi$.
    \end{lemma}
    We now present the following a priori energy estimate.
 \begin{proposition}\label{esti;energy esti}
     Let $\eta>\frac 14$, $\kappa\in(2,2\eta+\frac 32)$ and $u_0\in \dot{H}^1(\m T)$ satisfying $\int_{\m T}u(x)\,\mathrm{d}x=0$. If $v\in L^2([0,T];\dot{H}^{\frac 32}(\m T))$ is the solution of \eqref{v} with $X\in C([0,T];W^{\kappa-1})$, then we have the following energy estimate:
     \begin{align}\label{eq;energy estimate}
         \Vert v(t)\Vert^2_{H^1}+\int_0^t\|v(t')\|_{H^2} ^2\,\mathrm{d}t'\leq e^{(C\|X\|^2_{C_T(H^{\kappa-1})}+2)t}\left(\|u_0\|_{H^1}^2+ C\|X\|_{C_T(H^{\kappa-1})}^4t\right).
     \end{align}
 \end{proposition}
     
     \begin{proof}
         Define $y=v-v_{xx}$, we can deduce that 
         \begin{equation}
             \p_t y-\Delta y+vy_x+2yv_x=f(v,X)-f_{xx}(v,X),
         \end{equation}
         where $f(v,X)\triangleq-vX_x-Xv_x-XX_x-\partial_x(1-\partial_{xx})^{-1}(2vX+X^2
     +v_xX_x+\frac{1}{2}X_x^2)$. Multiplying $2v$ both sides and integrating, we have
     \begin{align*}
       2\int_{\m T}y_t v\,\mathrm{d}x=2\int_{\m T}\Delta y v\,\mathrm{d}x-2\int_{\m T}vy_xv\,\mathrm{d}x-4\int_{\m T}yv_xv\,\mathrm{d}x+2\int_{\m T}(f(v,X)-f_{xx}(v,X))v\,\mathrm{d}x.
     \end{align*}
     Using that $2\int_{\m T}vy_xv\,\mathrm{d}x+4\int_{\m T}yv_xv\,\mathrm{d}x=0$, we have
     \begin{equation}\label{eq;energy conservation}
         2\int_{\m T}y_t v\,\mathrm{d}x=2\int_{\m T}\Delta y v\,\mathrm{d}x+2\int_{\m T}(f(v,X)-f_{xx}(v,X))v\,\mathrm{d}x.
     \end{equation}
     Noting that $\Vert v\Vert_{H^1}^2=\int_{\m T}(v^2+v_x^2)\,\mathrm{d}x=\int_{\m T}yv\,\mathrm{d}x$, combining with \eqref{eq;energy conservation}, we have
     \begin{align}\label{eq;dt energy}
         \frac{\mathrm{d}}{\mathrm{d}t}\Vert v\Vert_{H^1}^2+2\Vert v\Vert_{H^2}^2=2\Vert v\Vert_{L^2}^2+2\int_{\m T}(f(v,X)-f_{xx}(v,X))v\,\mathrm{d}x .
     \end{align}
    For the term $\int_{\m T}(f(v,X)-f_{xx}(v,X))v\,\mathrm{d}x$, a straightforward calculation gives 
     \begin{align}\label{eq;f-fxx}
         &\int_{\m T}\left(f(v,X)-f_{xx}(v,X)\right)v \,\mathrm{d}x\notag\\
         &=\int_{\m T}\Big(-v^2 X_x-vv_x X-v XX_x+v_x(1-\p_{xx})^{-1}(2v X+X^2
     +v_xX_x+\frac{1}{2}X_x^2)+f_xv_x \Big)\,\mathrm{d}x \notag\\
         &=\int_{\m T} \left( -\frac 32v^2 X_x+\frac 32 v_x X^2-\frac 12 v_x^2 X_x -vv_xX_{xx}+v_{xx}XX_x+\frac 12 v_xX_x^2 \right)\,\mathrm{d}x.
     \end{align}
     Noting that $v\in L^2([0,T];\dot{H}^{\frac 32}(\m T))$ and $ X\in C([0,T];W^{\kappa-1})$, using Bony’s decomposition, we obtain
     \begin{align*}
         -vv_x-\p_x(vX)-XX_x-\partial_x(1-\partial_{xx})^{-1}\left(v^2+2vX+X^2
     +\frac{1}{2}v_x^2+v_xX_x+\frac{1}{2}X_x^2\right)\in L^2_T(L^2).
     \end{align*}
     Applying Lemma \ref{le;v_estimates} with $s=1$ implies that 
     $$v\in C([0,\bar{T}];\dot{H}^1(\m T))\cap L^2([0,\bar{T}];\dot{H}^2(\m T)).$$
     It is straightforward to verify that all the above integrations by parts make sense.
     
     Now we turn to the estimate of \eqref{eq;f-fxx}. The estimates for the first two terms of \eqref{eq;f-fxx} follow by similar arguments, so we focus here on estimating the remaining terms. Since $X\in C_TW^{\kappa-1}(\m T)$, we have 
     \begin{align}
         &\int_{\m T}v_x^2 X_x \,\mathrm{d}x
          \leq \|v_x\|_{H^1}\|v_x\|_{L^2}\|X\|_{H^{\kappa-1}}
          \leq \frac{1}{12} \|v_x\|_{H^1}^2 +C \|v\|_{H^1}^2\|X\|_{H^{\kappa-1}}^2,\label{eq1}\\
         &\int_{\m T}vv_xX_{xx} \,\mathrm{d}x
          \leq \|v\|_{H^1}\|v_x\|_{H^1}\|X\|_{H^{\kappa-1}} 
          \leq \frac{1}{12} \|v_x\|_{H^1}^2 +C \|v\|_{H^1}^2\|X\|_{H^{\kappa-1}}^2,\label{eq2}\\
         &\int_{\m T}v_{xx}XX_x \,\mathrm{d}x\leq  \|v_x\|_{H^1}\|X\|_{H^{\kappa-1}}^2 
          \leq \frac{1}{12} \|v_x\|_{H^1}^2 +C \|X\|_{H^{\kappa-1}}^4   \label{eq3} 
     \end{align}  
     and
        \begin{align}\label{eq4}
           \int_{\m T}v_xX_x^2 \,\mathrm{d}x&\leq \|X_x\|_{H^{k-2}}\|v_xX_x\|_{H^{2-k}} \notag\\
          &\leq \|X\|_{H^{\kappa-1}}(\|v_x\|_{L^\infty}\|X_x\|_{H^{2-k}}+\|X_x\|_{B^{-1}_{\infty,\infty}}\|v_x\|_{H^{3-k}}+\|v_x\|_{B^{-\frac 14}_{\infty,\infty}}\|X_x\|_{L^2} )\notag\\
          &\leq \|v_x\|_{H^1}\|X\|_{H^{\kappa-1}}^2 \notag\\
          &\leq \frac{1}{12} \|v_x\|_{H^1}^2 +C \|X\|_{H^{\kappa-1}}^4. 
        \end{align}
     Integrating equation \eqref{eq;dt energy} with respect to time over $[0,t]$, and combining inequalities \eqref{eq;f-fxx}, \eqref{eq1}, \eqref{eq2}, \eqref{eq3} and \eqref{eq4}, we obtain 
     \begin{align*}
         &\Vert v(t)\Vert^2_{H^1}+\int_0^t\|v(t')\|_{H^2}^2\,\mathrm{d}t' \\
         &\leq \|u_0\|_{H^1}^2+ C\int^t_0\|X(t')\|_{H^{\kappa-1}}^4\mathrm{d}t'+\int_0^t\Vert v(t')\|_{H^1}^2(C\|X(t')\|_{H^{\kappa-1}}^2+2)\,\mathrm{d}t'.
     \end{align*}
    Applying Gronwall’s inequality yields 
    \begin{align*}
         \Vert v(t)\Vert^2_{H^1}+\int_0^t\|v(t')\|_{H^2}^2\,\mathrm{d}t' \leq e^{(C\|X\|^2_{C_T(H^{\kappa-1})}+2)t}\left(\|u_0\|_{H^1}^2+ C\|X\|_{C_T(H^{\kappa-1})}^4t\right),
     \end{align*}
    which completes the proof.
    \end{proof}

\subsection{Proof of Theorem \ref{thm;global solution}}
 \begin{lemma}\label{le;exist and blow-up}
        Let $\eta>\frac 14, s\in[1,2\eta+\frac 12), u_0\in \dot{H}^s(\m T)$ satisfying $\int_{\m T}u(x)\,\mathrm{d}x=0$. Let $\kappa\in(2,2\eta+\frac 32)$ for $s=1$, and $\kappa\in[s+1,2\eta+\frac 32)$ for $s>1$. If there exists a constant $C$ such that
        \begin{equation}\label{eq;condition_X}
            CT_0^{\frac{1}{4}}\Vert X\Vert_{C_{\bar{T}}W^{\kappa-1}}<\frac{1}{4},\quad\text{for}\quad 0<T_0\leq \bar{T},
        \end{equation}
        then there exists $T\leq T_0$ such that \eqref{v} has a unique solution $v\in L^4_T([0,T];\dot{H}^{s+\frac 12})$. Let $T_{u_0}$ denote the maximal time of existence of such a solution. Then:
    \begin{itemize}
    \item For the same constant $C$,
    \begin{equation*}
        \Vert u_0\Vert_{\dot{H}^s}\leq\frac{1}{16C}\Longrightarrow T_{u_0}\geq T_0.
    \end{equation*}
    \item If $T_{u_0}\leq T_0$, then
    \begin{equation*}
        \int^{T_{u_0}}_0\Vert v(t)\Vert_{\dot{H}^{s+\frac 12}}^4 \,\mathrm{d}t=\infty.
    \end{equation*}
    \end{itemize}
    \end{lemma}
    
    \begin{proof}
        For equation \eqref{v}, we know that the solution $v$ has the form
        \begin{equation*}
          v(t,x)=P(t)u_0-\left[\frac{1}{2}\mathcal{B}(v,v)+\mathcal{B}(v,X)+\frac{1}{2}\mathcal{B}(X,X)\right],
    \end{equation*}
    where $\mathcal{B}(f,g)$ is defined by equation \eqref{def;B}. We note that $\mathcal{B}(f,g)$ satisfies 
    \begin{align*}
        \p_t\mathcal{B}(f,g)-\Delta\mathcal{B}(f,g)=\p_x(fg)+2\p_x(1-\p_{xx})^{-1}(fg)+\p_x(1-\p_{xx})^{-1}(f_xg_x).
    \end{align*}
    By virtue of Lemma \eqref{le;v_estimates} and Bony’s decomposition, we obtain the following estimates:
    \begin{gather}
        \Vert P(t)u_0-\frac 12 \mathcal{B}(X,X)\Vert_{L^4_T(\dot{H}^{s+\frac 12})}\leq \left( \Vert u_0\Vert_{\dot H^{s}}+CT^{\frac 12}\Vert X\Vert_{C_TW^{\kappa-1}}^2\right), \\
        \Vert \mathcal{B}(v,X)\Vert_{L^4_T(\dot{H}^{s+\frac 12})}\leq CT^{\frac 14}\Vert v\Vert_{L_T^4\dot H^{s+\frac 12}}\Vert X\Vert_{C_TW^{\kappa-1}}, \\
        \Vert  \mathcal{B}(v,v)\Vert_{L^4_T(\dot{H}^{s+\frac 12})}\leq C\Vert v\Vert_{L_T^4\dot H^{s+\frac 12}}^2.
    \end{gather}
    Applying Lemma \ref{le;continuous dependence} with $r=\frac{1}{8C}$ yields that if $\Vert u_0\Vert_{\dot{H}^s}\leq\frac{1}{16C}$ and $CT_0^{\frac{1}{4}}\Vert X\Vert_{C_{\bar{T}}W^{\kappa-1}}<\frac{1}{4}$, there exists a unique solution of \eqref{v} in the ball $B(0,2r)$ in the space $L^4([0,T_0];\dot H^{s+\frac 12})$.

    Now we consider the case of general initial data. Since $u_0\in \dot{H}^s(\m T)$, there exists $\rho_{u_0}$ such that
    \begin{align*}
        \left( \sum_{|k|\geq \rho_{u_0}}|k|^{2s}|\widehat{u_0}(k)|^2\right)^{\frac{1}{2}}\leq \frac{1}{32C}.
    \end{align*}
    Setting $u_0^\flat=\mathcal{F}(\mathbf{1}_{B(0,\rho_{u_0})}\widehat{u_0})$, by Lemma \ref{le;v_estimates}, if $T\leq \left(\frac{1}{32C\rho_{u_0}^{\frac 12}\Vert u_0\Vert_{\dot H^s}} \right)^{\frac 14} \wedge T_0$, we have
    \begin{align*}
        \Vert P(t)u_0\Vert_{L^4_T(\dot{H}^{s+\frac 12})}&\leq \Vert u_0-u_0^\flat\Vert_{\dot H^s}+\Vert P(t)u_0^\flat\Vert_{L^4_T(\dot{H}^{s+\frac 12})} \\
        &\leq \frac{1}{32C}+\rho_{u_0}^{\frac 12}T^\frac{1}{4}\Vert u_0\Vert_{\dot H^s} \\
        &\leq \frac{1}{16C}.
    \end{align*}
    Using an argument similar to the above case implies that there exists a unique solution of \eqref{v} in the space $L^4([0,T];\dot H^{s+\frac 12})$.
    
    Finally, we prove the blow-up criterion by means of a contradiction argument. Assume that $v$ is a solution of \eqref{v} on $[0,T)$ satisfying 
    \begin{equation*}
        \int^T_0\Vert v(t)\Vert^4_{\dot H^{s+\frac 12}}\,\mathrm{d}t<\infty.
    \end{equation*}
    Denote by $T_{u_0}$ the maximal time of existence. We claim that $T_{u_0}>T$. Indeed, by Lemma \ref{le;v_estimates} and condition \eqref{eq;condition_X}, we have
    \begin{align*}
        &\left(\sum_{k\neq 0}|k|^{2s}\left(\sup_{t\in(0,T)}\left\vert \hat v(t,k)\right\vert\right)^2\right)^\frac{1}{2}\\ &\leq \Vert u_0\Vert_{\dot H^s} +\frac{1}{\sqrt{2}}\Big\Vert -\frac 12\p_x(v^2)-\p_x(vX)-\frac 12\p_x(X^2)-\partial_x(1-\partial_{xx})^{-1}\Big(v^2+2vX+X^2\\
     &~~~~~~~~~~~~~~~~~~~~~~~~~~~+\frac{1}{2}v_x^2+v_xX_x+\frac{1}{2}X_x^2\Big)\Big\Vert_{L_T^2(\dot H^{s-1})}\\
     &\leq \Vert u_0\Vert_{\dot H^s} +\frac{C}{\sqrt{2}}\left( \Vert v\Vert_{L^4_T(\dot{H}^{s+\frac 12})}^2+ T^{\frac 14}\Vert v\Vert_{L^4_T(\dot{H}^{s+\frac 12})}\Vert X\Vert_{C_TW^{\kappa-1}}+T^{\frac 12}\Vert X\Vert_{C_TW^{\kappa-1}}^2\right) \\
     &\leq \Vert u_0\Vert_{\dot H^s} +\frac{C}{\sqrt{2}}\Vert v\Vert_{L^4_T(\dot{H}^{s+\frac 12})}^2+C\Vert v\Vert_{L^4_T(\dot{H}^{s+\frac 12})}+C <\infty,
    \end{align*}
    which implies that there exists $\rho>0$ such that 
    \begin{align*}
        \left( \sum_{|k|\geq \rho}|k|^{2s}|\widehat{v}(t,k)|^2\right)^{\frac{1}{2}}\leq \frac{1}{32C},\quad\forall t\in[0,T).
    \end{align*}
    We can therefore deduce that $T_{u_0}>T$, which completes the proof.
    \end{proof}
    
Theorem \ref{thm;global solution} is established by combining Proposition \ref{esti;energy esti} and Lemmas \ref{le;v_estimates}, \ref{le;exist and blow-up}.
\begin{proof}[Proof of Theorem \ref{thm;global solution}]
    For $X$ satisfying \eqref{eq;X(t) decomposition}, using Lemma \ref{X_bdd}, we have 
    \begin{equation*}
        X\in C([0,\bar{T}];W^{\kappa-1}),\quad \forall\, \bar T\geq 0, \kappa\in(2,2\eta+\frac 32).
    \end{equation*}
    Fix $\bar T\geq 0$ and $\kappa\in(2,2\eta+\frac 32)$. It follows that there exists $T_0\in(0,\bar T]$ such that 
    \begin{equation*}
        CT_0^{\frac{1}{4}}\Vert X\Vert_{C_{\bar{T}}W^{\kappa-1}}<\frac{1}{4}.
    \end{equation*}
    Then, Lemma \ref{le;exist and blow-up} implies that  there exists $\widetilde{T}\leq T_0$ such that \eqref{v} has a unique solution $v\in L^4([0,\widetilde{T}];\dot H^{\frac 32} )$. Let $T_{u_0}$ denote the maximal time of existence of such a solution. We claim that $T_{u_0}>T_0$. Indeed, if $T_{u_0}\leq T_0$, using Lemma \ref{le;exist and blow-up} implies that
    \begin{align}\label{eq;infty}
        \int^{T_{u_0}}_0\Vert v(t)\Vert_{\dot{H}^{\frac 32}}^4 \,\mathrm{d}t=\infty.
    \end{align}
    It follows from interpolation inequality and the energy estimate from Proposition \ref{esti;energy esti} that for all $T<T_{u_0}$, we have 
    \begin{equation*}
        \Vert v\Vert_{L^4_T\dot H^{\frac 32}}\leq \Vert v\Vert_{C_T\dot H^{1}}^{\frac 12} \Vert v\Vert_{L_T^2\dot H^{2}}^{\frac 12}\leq C(\Vert u_0\Vert_{\dot H^{1}},  \Vert X\Vert_{C_{\bar T} W^{\kappa-1}}, T_{u_0})<\infty
    \end{equation*}
    with uniform upper bound, which  contradicts \eqref{eq;infty}. The claim is proved, and we now have $v\in C([0,T];\dot{H}^1(\m T))\cap L^2([0,T];\dot{H}^2(\m T))$. If $T_0<\bar T$, define $\widetilde{T_0}=T_0\wedge(\bar T-T_0)$ and take $T_0$ as the initial time. We also have 
    \begin{equation*}
        C\widetilde{T_0}^{\frac{1}{4}}\Vert X\Vert_{C([T_0,\bar{T}];W^{\kappa-1})}<\frac{1}{4}.
    \end{equation*}
    Let $T_{v(T_0)}$ be the maximal existence time of the solution starting from the initial time $T_0$. Repeating the previous argument, we obtain $\widetilde{T_0}<T_{v(T_0)}$. After a finite number of such arguments, we deduce that $v\in C([0,\bar{T}];\dot{H}^1(\m T))\cap L^2([0,\bar{T}];\dot{H}^2(\m T))$. Since $\bar T$ is arbitrary, we obtain $v$ as a global solution of \eqref{v}. By the proof of Lemma \ref{le;exist and blow-up}, it’s unique. Since $X$ satisfies \eqref{eq;X(t) decomposition}, we prove that $u=v+X $ is the global solution of \eqref{eq;CH;diffusion}. We complete the proof.

\end{proof}

 \appendix
\section{Fixed-point argument}

\begin{lemma}\label{le;Picard th}\label{le;continuous dependence}
	Let $E$ be a Banach space, $\ml B$ a continuous bilinear map from $E\times E$ to $E$, and $r$ a positive real number such that
	$$r<\frac{1}{8\Vert \ml B\Vert}\quad with \quad\Vert \ml B\Vert=\sup_{\Vert u\Vert_E,\Vert v\Vert_E\leq1}\Vert \ml B(u,v)\Vert_E.$$ 
	Let $L: E \rightarrow E$ be a linear map satisfying
    \begin{equation}\label{eq;condition L}
        \sup_{v\in E}\frac{\Vert  L(v)\Vert_E}{\Vert v\Vert_E}< \frac{1}{4}.
    \end{equation}
	Then for any $a$ in a ball $B(0,r)\subset E$, there exists a unique $\tilde x$ in $\ml B(0,2r)$ such that 
	$$\tilde x=a+ L(\tilde x)+\ml B(\tilde x,\tilde x).$$
    Moreover, if there exists $a_n$ and $L_n$ satisfying
       \begin{equation}\label{eq;an,Ln}
        \Vert a_n-a\Vert_E\to 0,\quad \Vert L_n-L\Vert_{E\to E},\quad as\: n\to \infty.
    \end{equation}
    then for $x_n=a_n+L_n(x)+\ml B(x_n,x_n)$, there exists $N\in \m N$ such that for any $n>N$, we have
    \begin{equation}
        \Vert x_n-x\Vert_E\leq 4\Vert a_n-a\Vert_E+8r\Vert L_n-L\Vert_{E\to E}.
    \end{equation}
\end{lemma}

We also record a variant of the abstract fixed-point argument that will be 
useful for the coupled systems considered below. Let
\[
    f:E\times E\to X
\]
be a linear map, that is,
\begin{equation}\label{eq;linear;f}
    \begin{aligned}
        f(c(x,y))&=c f(x,y),\\
        f(x_1+x_2,y_1+y_2)
        &=f(x_1,y_1)+f(x_2,y_2).
    \end{aligned}
\end{equation}
Assume moreover that $f$ satisfies
\begin{equation}\label{eq;controlled;paracontrolled ansatz}
    \|f(x,y)\|_X
    \leq \|x\|_{\dot E}+c\|y\|_E.
\end{equation}
We then have the following fixed-point result.
\begin{lemma}\label{le;Picard th;specific}
		Let $\dot E, E,X$ be Banach spaces satisfying $ \dot E\subset X \subset E$, $
		\ml E:=\dot E\times E$ be a Banach space with norm 
		$$\Vert (x,y)\Vert_{\ml E}=2\Vert x\Vert_{\dot E}+\Vert y\Vert_E.$$
		Moreover, let $\ml B$ be a continuous bilinear map from $X\times X$ to $\dot E$, and $r$ a positive real number such that
		$$\Vert \ml B(u,v)\Vert_{\dot  E}< \frac{1}{16r}\Vert u\Vert_X \Vert v\Vert_X.$$
		Let $\ml G:X\to\dot E$, $g:E\to \dot E$ and $h:\dot E\to\dot E$ be linear maps  satisfying 
		$$\sup_{v\in E}(\frac{\Vert \ml G(v)\Vert_{\dot E}}{\Vert v\Vert_X},\frac{\Vert  g(v)\Vert_{\dot E}}{\Vert v\Vert_E})<\frac{1}{16},\quad \frac{\Vert h(v)\Vert_{\dot E}}{\Vert v\Vert_{\dot E}}< \frac{1}{16}.$$
		Let $f$ satisfy \eqref{eq;linear;f} and \eqref{eq;controlled;paracontrolled ansatz} for $X$. Then, for any $(a,b)\in B(0,\frac{r}{4})\subset \ml E$ and $0<c<\frac{1}{30}$, there exists a unique $(x,y)\in B(0,4r)\subset \ml E$ such that 
		\begin{equation}\label{eq;key mapping}
			y=b+f(x,y),\quad  x=a+g(y)+h(x)+\ml B(f(x,y),f(x,y))+\ml G(f(x,y)).
		\end{equation}
\end{lemma}
Let $f_n$ also satisfy \eqref{eq;linear;f} and \eqref{eq;controlled;paracontrolled ansatz} with the following property
\begin{equation}\label{eq;f_n}
    f_n(x,y)-f(x,y)=(f_n-f)(0,y), \quad \Vert (f_n-f)(x,y)\Vert_X\leq \Vert x\Vert_{\dot E}+(c_n-c)\Vert y\Vert _E.
\end{equation}

\begin{lemma}\label{le:Picard-stability}
    Let the assumptions of Lemma~\ref{le;Picard th;specific} hold.
    For each $n\geq1$, let
    \[
        f_n:E\times E\to X,\qquad
        g_n:E\to\dot E,\qquad
        h_n:\dot E\to\dot E,\qquad
        \mathcal G_n:X\to\dot E
    \]
    satisfy the same assumptions, uniformly in $n$, with constants
    $c_n<1/30$. Assume that
    \begin{equation}
        (f_n-f)(x,y)=(f_n-f)(0,y)
    \end{equation}
    and
    \begin{equation}\label{eq;fn-difference}
        \|(f_n-f)(0,y)\|_X
        \leq |c_n-c|\,\|y\|_E.
    \end{equation}
    Set
    \begin{equation}\label{eq:Delta-Xn}
        \Delta_n\mathbb X
        :=
        \|g_n-g\|_{E\to\dot E}
        +\|h_n-h\|_{\dot E\to\dot E}
        +\|\mathcal G_n-\mathcal G\|_{X\to\dot E}
        +|c_n-c|.
    \end{equation}
    Suppose that
    \[
        \Delta_n\mathbb X\longrightarrow0.
    \]
    Let $(a,b),(a_n,b_n)\in B(0,r/4)\subset\mathcal E$ and assume
    \[
        a_n\longrightarrow a\quad\text{in }\dot E,
        \qquad
        b_n\longrightarrow b\quad\text{in }E.
    \]
    Let $(x,y)$ and $(x_n,y_n)$ be the corresponding solutions given
    by Lemma~\ref{le;Picard th;specific}. Then
    \begin{equation}\label{eq:stability-fixed-point}
        2\|x_n-x\|_{\dot E}+\|y_n-y\|_E
        \lesssim
        \|a_n-a\|_{\dot E}
        +\|b_n-b\|_E
        +C(r,c)\Delta_n\mathbb X.
    \end{equation}
    In particular,
    \[
        (x_n,y_n)\longrightarrow(x,y)
        \qquad\text{in }\mathcal E.
    \]
    \begin{proof}
    By Lemma~\ref{le;Picard th;specific}, both $(x,y)$ and
    $(x_n,y_n)$ belong to $B(0,4r)\subset\mathcal E$. Hence
    \begin{equation}\label{eq:uniform-xy}
        \|x\|_{\dot E},\|x_n\|_{\dot E}\leq 2r,
        \qquad
        \|y\|_E,\|y_n\|_E\leq 4r.
    \end{equation}
    For simplicity, set
    \[
        A_n:=\|x_n-x\|_{\dot E},
        \qquad
        B_n:=\|y_n-y\|_E,
        \qquad
        \delta_n:=\Delta_n\mathbb X.
    \]

    We first estimate $B_n$. By the equations for $y_n$ and $y$,
    the linearity of $f_n$, and \eqref{eq;f_n},
    \begin{align*}
        B_n
        &\leq
        \|b_n-b\|_E
        +\|f_n(x_n,y_n)-f(x,y)\|_X\\
        &\leq
        \|b_n-b\|_E
        +\|f_n(x_n-x,y_n-y)\|_X
        +\|(f_n-f)(x,y)\|_X\\
        &=
        \|b_n-b\|_E
        +\|f_n(x_n-x,y_n-y)\|_X
        +\|(f_n-f)(0,y)\|_X.
    \end{align*}
    Therefore, by the control assumption on $f_n$,
    \begin{equation}\label{eq:Bn-estimate}
        B_n
        \leq
        \|b_n-b\|_E
        +A_n+c_nB_n
        +|c_n-c|\,\|y\|_E.
    \end{equation}
    Using \eqref{eq:uniform-xy}, we obtain
    \begin{equation}\label{eq:Bn-estimate2}
        (1-c_n)B_n
        \leq
        \|b_n-b\|_E+A_n+4r\delta_n.
    \end{equation}

    Next, define
    \[
        F_n:=f_n(x_n,y_n),
        \qquad
        F:=f(x,y).
    \]
    The uniform bounds \eqref{eq:uniform-xy} imply
    \begin{align}
        \|F_n\|_X
        &\leq 2r+4c_nr,\label{eq:Fn-bound}\\
        \|F\|_X
        &\leq 2r+4cr.\label{eq:F-bound}
    \end{align}
    Moreover, arguing as above,
    \begin{equation}\label{eq:Fn-F}
        \|F_n-F\|_X
        \leq
        A_n+c_nB_n+4r\delta_n.
    \end{equation}

    We now estimate $A_n$. Subtracting the fixed-point equations for
    $x_n$ and $x$, we obtain
    \begin{align*}
        x_n-x
        ={}&
        a_n-a
        +\big(g_n(y_n)-g(y)\big)
        +\big(h_n(x_n)-h(x)\big)\\
        &+\mathcal B(F_n,F_n)-\mathcal B(F,F)
        +\mathcal G_n(F_n)-\mathcal G(F).
    \end{align*}
    We estimate each term separately. First,
    \begin{align*}
        \|g_n(y_n)-g(y)\|_{\dot E}
        &\leq
        \|(g_n-g)(y_n)\|_{\dot E}
        +\|g(y_n-y)\|_{\dot E}\\
        &\leq
        4r\delta_n+\frac1{16}B_n.
    \end{align*}
    Similarly,
    \begin{align*}
        \|h_n(x_n)-h(x)\|_{\dot E}
        &\leq
        \|(h_n-h)(x_n)\|_{\dot E}
        +\|h(x_n-x)\|_{\dot E}\\
        &\leq
        2r\delta_n+\frac1{16}A_n.
    \end{align*}

    For the bilinear term, using bilinearity,
    \[
        \mathcal B(F_n,F_n)-\mathcal B(F,F)
        =
        \mathcal B(F_n-F,F_n)
        +\mathcal B(F,F_n-F).
    \]
    Hence
    \begin{align*}
        &\|\mathcal B(F_n,F_n)-\mathcal B(F,F)\|_{\dot E}\\
        &\qquad\leq
        \frac1{16r}
        \|F_n-F\|_X
        \big(\|F_n\|_X+\|F\|_X\big).
    \end{align*}
    Combining \eqref{eq:Fn-bound}--\eqref{eq:Fn-F}, we obtain
    \begin{align}
        &\|\mathcal B(F_n,F_n)-\mathcal B(F,F)\|_{\dot E}
        \notag\\
        &\qquad\leq
        \frac{1+c_n+c}{4}
        \big(A_n+c_nB_n+4r\delta_n\big).
        \label{eq:B-difference}
    \end{align}

    Finally,
    \begin{align*}
        \|\mathcal G_n(F_n)-\mathcal G(F)\|_{\dot E}
        &\leq
        \|\mathcal G_n(F_n-F)\|_{\dot E}
        +\|(\mathcal G_n-\mathcal G)(F)\|_{\dot E}\\
        &\leq
        \frac1{16}\|F_n-F\|_X
        +\delta_n\|F\|_X\\
        &\leq
        \frac1{16}
        \big(A_n+c_nB_n+4r\delta_n\big)
        +(2r+4cr)\delta_n.
    \end{align*}

    Collecting the preceding estimates gives
    \begin{align}
        A_n
        \leq{}&
        \|a_n-a\|_{\dot E}
        +\frac1{16}B_n+\frac1{16}A_n
        +C(r,c)\delta_n \notag\\
        &+
        \left(
            \frac{1+c_n+c}{4}+\frac1{16}
        \right)
        \big(A_n+c_nB_n+4r\delta_n\big).
        \label{eq:An-pre}
    \end{align}
    
  Combining the above estimates, we obtain
\begin{align*}
    2A_n+B_n
    \leq{}&
    2\|a_n-a\|_{\dot E}
    +\|b_n-b\|_E
    +C(r,c)\delta_n \\
    &+\theta_1 A_n+\theta_2 B_n,
\end{align*}
where $\theta_1$ and $\theta_2$ depend only on $c$ and $c_n$.
Since $c,c_n<1/30$, these coefficients are sufficiently small so that
\[
    \theta_1 A_n+\theta_2 B_n
    \leq \theta(2A_n+B_n)
\]
for some $\theta\in(0,1)$ independent of $n$. Hence, absorbing this term
into the left-hand side yields
\begin{equation}
    2\|x_n-x\|_{\dot E}+\|y_n-y\|_E
    \lesssim
    \|a_n-a\|_{\dot E}
    +\|b_n-b\|_E
    +C(r,c)\Delta_n\mathbb X.
\end{equation}
    which completes the proof.
\end{proof}
\end{lemma}

 \section{Some useful estimates}

\begin{lemma}\label{le;Ist}
For \(0\leq s\leq t\), define

$$
I_{s,t}(k)
:=
\int_0^t\int_0^s
e^{-(t-r)|k|^2-(s-\tau)|k|^2-|r-\tau||k|^2}
\,\mathrm d\tau\,\mathrm dr.
$$
Then, for every \(\delta\in[0,1]\),
\begin{align*}
|I_{t,t}(k)|
&\lesssim \frac{1}{|k|^4},\quad
|I_{t,t}(k)-I_{s,t}(k)|
\lesssim
\frac{|t-s|^\delta}{|k|^{4-2\delta}} .
\end{align*}

\begin{proof}
We first decompose
$$
I_{s,t}(k)
=
e^{-(t-s)|k|^2}I_{s,s}(k)
+
\int_s^t\int_0^s
e^{-(t-r)|k|^2-(s-\tau)|k|^2-(r-\tau)|k|^2}
\,\mathrm d\tau\,\mathrm dr .
$$
For the second term, a straightforward calculation gives
$$
\begin{aligned}
&\int_s^t\int_0^s
e^{-(t-r)|k|^2-(s-\tau)|k|^2-(r-\tau)|k|^2}
\,\mathrm d\tau\,\mathrm dr\\
&
=e^{-(t+s)|k|^2}
\int_s^t\int_0^s e^{2\tau|k|^2}\,\mathrm d\tau\,\mathrm dr
=
\frac{(t-s)\big(e^{-(t-s)|k|^2}-e^{-(t+s)|k|^2}\big)}
{2|k|^2}.
\end{aligned}
$$
On the other hand, by symmetry,
$$
\begin{aligned}
I_{s,s}(k)
&=
2\int_0^s\int_0^r
e^{-(s-r)|k|^2-(s-\tau)|k|^2-(r-\tau)|k|^2}
\,\mathrm d\tau\,\mathrm dr\\
&=
2e^{-2s|k|^2}
\int_0^s\int_0^r e^{2\tau|k|^2}\,\mathrm d\tau\,\mathrm dr\\
&=
\frac{1-e^{-2s|k|^2}}{2|k|^4}
-\frac{s e^{-2s|k|^2}}{|k|^2}.
\end{aligned}$$
Since \(s e^{-2s|k|^2}\lesssim |k|^{-2}\), it follows that
$$
|I_{s,s}(k)|\lesssim |k|^{-4}.
$$
Combining the above identities and using
$$
|1-e^{-x}|\lesssim x^\delta,
\quad
x^\delta e^{-x}\lesssim1,
\quad \delta\in[0,1],
$$
we obtain
$$
|I_{t,t}(k)-I_{s,t}(k)|
\lesssim
\frac{|t-s|^\delta}{|k|^{4-2\delta}},
$$
which completes the proof.
\end{proof}
\end{lemma}

\begin{lemma}\label{le;covariance; QX}
    The spatial Fourier transform $\widehat{\ml Q}(t,k)$ is a Gaussian process with zero-mean satisfying
         \begin{align*}
    \m E[\widehat {\ml Q}(t,k)\widehat {\ml Q}(s,l)]=\textit{1}_{k+l=0}\frac{1}{(1+|k|^2)^{2\eta}} I_{s,t}(k).
            \end{align*}
        Moreover, we have the following estimates
        \begin{align*}
          \m E[\widehat {Y}(t,k)\widehat {Y}(t,l)]\lesssim\frac{1}{(1+|k|^2)^{2\eta}|k|^{2}},\quad \m E[\widehat {\ml Q}(t,k)\widehat {\ml Q}(t,l)]\lesssim\frac{1}{(1+|k|^2)^{2\eta}|k|^{4}}\\
            \m E[\widehat {\ml Q}_{s,t}(k)\widehat {\ml Q}_{s,t}(l)],  \m E[\widehat {\ml Q}_{s,t}(k)\widehat {\ml Q}(s,l)\lesssim\frac{(t-s)^\delta|k|^{2\delta}}{(1+|k|^2)^{2\eta}|k|^{4}}\\
             \m E[\widehat {Y}_{s,t}(k)\widehat {Y}_{s,t}(l)],    \m E[\widehat {Y}_{s,t}(k)\widehat {Y}(s,l)]\lesssim\frac{(t-s)^\delta|k|^{2\delta}}{(1+|k|^2)^{2\eta}|k|^{2}}
        \end{align*}
and for $0\leq s\leq t$, we have
\begin{align*}
         \m E[\widehat {\ml Q}(s,k)\widehat {Y}_{s,t}(k)],\:\m E[\widehat {Y}(t,k)\widehat {\ml Q}_{s,t}(k)]\lesssim \frac{|t-s|^\delta|k|^{2\delta}}{(1+|k|^2)^{2\eta}|k|^3}.
\end{align*}

\begin{proof}
We first estimate the mixed covariance between \(Y\) and \(\ml Q\). A direct calculation gives
$$
\begin{aligned}
\left|
\mathbb E[
\widehat Y(t,k)\widehat{\ml Q}(s,l)]
\right|
&=
\mathbf 1_{\{k+l=0\}}
\frac{1}{(1+|k|^2)^{2\eta}|k|}
\int_0^s
e^{-(s-r)|k|^2-(t-r)|k|^2}\,\mathrm dr\\
&=
\mathbf 1_{\{k+l=0\}}
\frac{
e^{-(t-s)|k|^2}-e^{-(t+s)|k|^2}
}{
2(1+|k|^2)^{2\eta}|k|^3
}\\
&\lesssim
\mathbf 1_{\{k+l=0\}}
\frac{1}{(1+|k|^2)^{2\eta}|k|^3}.
\end{aligned}
$$
Using \(|1-e^{-x}|\lesssim x^\delta\), we therefore obtain
$$
\left|
\mathbb E[
\widehat Y_{s,t}(k)\widehat{\ml Q}(s,l)]
\right|
\lesssim
\mathbf 1_{\{k+l=0\}}
\frac{|t-s|^\delta|k|^{2\delta}}
{(1+|k|^2)^{2\eta}|k|^3}.
$$
A similar argument gives
$$
\left|
\mathbb E[
\widehat Y(t,k)\widehat{\ml Q}_{s,t}(l)]
\right|
\lesssim
\mathbf 1_{\{k+l=0\}}
\frac{|t-s|^\delta|k|^{2\delta}}
{(1+|k|^2)^{2\eta}|k|^3}.
$$

We next consider the covariance of \(\widehat{\ml Q}\). By a direct calculation,
$$
\mathbb E[
\widehat{\ml Q}(t,k)\widehat{\ml Q}(s,l)]
=
\mathbf 1_{\{k+l=0\}}
\frac{I_{s,t}(k)}
{(1+|k|^2)^{2\eta}}.
$$
Lemma \ref{le;Ist} then yields
$$
\left|
\mathbb E[
\widehat{\ml Q}(t,k)\widehat{\ml Q}(t,l)]
\right|
\lesssim
\mathbf 1_{\{k+l=0\}}
\frac{1}{(1+|k|^2)^{2\eta}|k|^4}.
$$
Moreover,
$$
\begin{aligned}
\left|
\mathbb E[
\widehat{\ml Q}_{s,t}(k)\widehat{\ml Q}(s,l)]
\right|
&=
\mathbf 1_{\{k+l=0\}}
\frac{|I_{s,t}(k)-I_{s,s}(k)|}
{(1+|k|^2)^{2\eta}}\\
&\lesssim
\mathbf 1_{\{k+l=0\}}
\frac{|t-s|^\delta}
{(1+|k|^2)^{2\eta}|k|^{4-2\delta}}.
\end{aligned}
$$

Finally, expanding the increment covariance and applying the preceding estimates gives
$$
\left|
\mathbb E[
\widehat{\ml Q}_{s,t}(k)
\widehat{\ml Q}_{s,t}(l)]
\right|
\lesssim
\mathbf 1_{\{k+l=0\}}
\frac{|t-s|^\delta}
{(1+|k|^2)^{2\eta}|k|^{4-2\delta}}.
$$
The corresponding estimates for \(Y\) follow in the same way from its covariance formula. This completes the proof.
\end{proof}
\end{lemma}

 \begin{lemma}\label{le;commutater;D}\cite{bahouriFourierAnalysisNonlinear2011}
	Let $f$ be a smooth function on $\mathbb{R}^d$ and be homogeneous of degree $m$ away from $0$. Let $\rho\in(0,1)$, $s\in\mathbb{R}$ and $p,r \in[1,\infty]$. If $p_1,p_2\in[1,\infty]$ such that $\frac{1}{p_1}+\frac{1}{p_2}=\frac{1}{p}$, then the bound
	\begin{align*}
		\Vert a\prec f(D)u-f(D)(a\prec u)\Vert_{B_{p,r}^{s-m+\rho}}\lesssim \Vert \nabla a\Vert_{B_{p_1,\infty}^{\rho-1}}\Vert u\Vert_{B_{p_2,r}^s}
	\end{align*}
	holds.
\end{lemma}

\begin{lemma}\label{le;commutator;resonant}\cite{zhang2026wellposednessregularityfractionalrough}
	Assume that $\alpha\in(0,1)$, $\eta,\gamma\in \mathbb{R}$ satisfying $\eta+\gamma<0$ and $\alpha+\eta+\gamma>0$, let $C(f,g,h)=(f\prec g)\circ h-f(g\circ h)$, then we have the following estimate
	$$\Vert C(f,g,h)\Vert_{H^{\alpha+\eta+\gamma}}\leq \Vert f\Vert_{\ml C^\alpha}\Vert g\Vert_{H^\eta}\Vert v\Vert_{\ml C^\gamma}.$$
	\begin{proof}
		We rewrite 
		\begin{align*}
			C(f,g,h)=\sum_{j,k\geq -1}\sum_{\vert i-j\vert\leq 1}[1_{i\gtrsim k}R_i(\Delta_k f,g)\Delta_j h-1_{i\leq k-N}\Delta_k f\Delta_jg\Delta_jh]	.
		\end{align*}
		It suffices to estimate above two terms. For any fixed $k$, we have the estimate 
		\begin{align*}
			\Vert \sum_{j\geq -1}\sum_{\vert i-j\vert\leq 1}1_{i\leq k-N}\Delta_k f\Delta_jg\Delta_jh\Vert_{L^2}\lesssim&2^{-k\alpha}\Vert f\Vert_{\ml C^\alpha}\sum_{i=-1}^{k-N}2^{-i(\eta+\gamma)}\Vert g\Vert_{H^\eta}\Vert h\Vert_{\ml C^\gamma}\\
			\lesssim&2^{-k(\alpha+\eta+\gamma)}\Vert f\Vert_{\ml C^\alpha}\Vert g\Vert_{H^\eta}\Vert h\Vert_{\ml C^\gamma},
		\end{align*}
		where we use the fact $\eta+\gamma<0$. Note that the Fourier transform of $$\sum_{j,k\geq -1}\sum_{\vert i-j\vert\leq 1}1_{i\gtrsim k}R_i(\Delta_k f,g)\Delta_j h$$
        is supported in the ball $2^k\ml B$. By Lemma 2.84 in \cite{bahouriFourierAnalysisNonlinear2011}, we obtain the estimate since $\alpha+\eta+\gamma>0$. For the first series, we can check that the Fourier transform of $$\sum_{k\geq -1}\sum_{\vert i-j\vert\leq 1}1_{i\gtrsim k}R_i(\Delta_k f,g)\Delta_j h$$
        is supported in a ball $2^j\ml B$ for fixed $j$. And we calculate that
		\begin{align*}
			\Vert \sum_{k\geq -1}\sum_{\vert i-j\vert\leq 1}1_{i\gtrsim k}R_i(\Delta_k f,g)\Delta_j h\Vert_{L^2}=&\Vert \sum_{\vert i-j\vert\leq 1}R_i(\sum_{1\leq k\lesssim i}\Delta_k f,g)\Delta_j h\Vert_{L^2}\\
			\lesssim &\sum_{\vert i-j\vert\leq 1}2^{-i(\eta+\gamma)}\Vert \sum_{k\lesssim i}\Delta_k f\Vert_{L^\infty} \Vert g\Vert_{H^\eta} 2^{-j\gamma}\Vert h\Vert_{\ml C^\gamma}\\
			\lesssim &2^{-j(\alpha+\eta+\gamma)}\Vert f\Vert_{\ml C^\alpha}\Vert g\Vert_{H^\eta}\Vert h\Vert_{\ml C^\gamma}.
		\end{align*}
		Combining all estimates above, we finish our proof.
	\end{proof}
\end{lemma}

\begin{lemma}\label{le;commutator;modfied paraproduct, L}\cite{zhang2026wellposednessregularityfractionalrough}
    Let $\alpha\in(0,2)$ and $\eta\in\mathbb{R}$. Then for $t>0$, we have 
		\begin{equation}\label{eq;commutator;modfied paraproduct, L;1}
			\Vert f\pprec g-f\prec g\Vert_{W^{\alpha+\eta}}\lesssim \Vert f\Vert_{C^{\frac{\alpha}{2}}_t L^{\infty}}\Vert g\Vert_{W^\eta}.
		\end{equation}
    Moreover, define operator $\ml L=\p_t-\Delta$. Then, for $\alpha\in(0,1)$, we have 
		\begin{equation}\label{eq;commutator;modfied paraproduct, L;2}
			\Vert \ml L(f\pprec g)(t)-(f\pprec\ml Lg)(t)\Vert_{W^{\alpha+\eta-2}}\lesssim \Vert f\Vert_{\widetilde{W}^\alpha_\infty(t)}\Vert g\Vert_{W^\eta}.
		\end{equation}
\end{lemma}
\bibliographystyle{alpha}
	\bibliography{ref}
\end{document}